\documentclass{lms}

\usepackage{amsmath, amsfonts, amscd}
\usepackage{graphics} \usepackage{color} \usepackage{epsfig} \usepackage[all]{xy}
\usepackage{graphicx,amssymb}

\newtheorem{theorem}{Theorem}[section] 
\newtheorem{lemma}[theorem]{Lemma}     
\newtheorem{corollary}[theorem]{Corollary}
\newtheorem{proposition}[theorem]{Proposition}

\newnumbered{conjecture}[theorem]{Conjecture}  
\newnumbered{definition}[theorem]{Definition}
\newnumbered{remark}[theorem]{Remark}
\newnumbered{claim}[theorem]{Claim}
\newnumbered{question}[theorem]{Question}
\newnumbered{example}[theorem]{Example}

\newtheorem{thmx}{Theorem}

\newcommand{\Ozsvath}{{Ozsv{\'a}th} }
\newcommand{\Szabo}{{Szab{\'o}} }

\newcommand{\N}{\ensuremath{\mathbb{N}}} 
\newcommand{\Z}{\ensuremath{\mathbb{Z}}}
\newcommand{\Q}{\ensuremath{\mathbb{Q}}}

\newcommand{\s}{\ensuremath{\mathfrak{s}}}
\newcommand{\e}{\ensuremath{\eta}}
\newcommand{\la}{\ensuremath{\lambda}}

\newcommand{\g}{\ensuremath{\gamma}}
\newcommand{\cc}{\mathcal{C}}

\DeclareMathOperator{\Spc}{Spin^c}
\DeclareMathOperator{\fr}{fr}
\DeclareMathOperator{\di}{d}

\title[Knot traces and concordance]
 {Knot traces and concordance} %

\author{A. N. Miller and L. Piccirillo }

\classno{57M25}

\extraline{The second author was supported by an NSF graduate research fellowship.}

\begin{document}
\bibliographystyle{alpha}

\maketitle

\begin{abstract}
We give a method for constructing many pairs of distinct knots $K_0$ and $K_1$ such that the two 4-manifolds obtained by attaching a 2-handle to $B^4$ along $K_i$ with framing zero are diffeomorphic.  We use the d-invariants of Heegaard Floer homology to obstruct the smooth concordance of some of these $K_0$ and $K_1$, thereby disproving a conjecture of Abe. As a consequence, we obtain a proof that there exist patterns $P$ in solid tori such that $P(K)$ is not always concordant to $P(U) \# K$ and yet whose action on the smooth concordance group is invertible. 
\end{abstract}

\section{Introduction}

\begin{conjecture}[(Akbulut-Kirby Conjecture, 1978. Problem 1.19 of \cite{Kir97})]
If $K$ and $K'$ have homeomorphic 0-surgeries,  then $K$ and $K'$ are smoothly concordant.  
\end{conjecture}
One might view this conjecture as motivated by the following two ideas. First, the  0-surgery of a knot determines fundamental concordance invariants such as the Tristram-Levine signature, the Alexander polynomial, and the algebraic concordance class of a knot, as well as more involved invariants such as Casson-Gordon signatures, metabelian twisted Alexander polynomials, and those associated to the filtration of \cite{COT03}. In fact, just the existence of a homology cobordism between the 0-surgeries of two knots implies that the algebraic concordance classes of those knots agree. 
 Secondly and perhaps more fundamentally, it is a well-known result (see Question 1.19 of \cite{Kir97}, \cite{AJOT13}) that assuming the smooth 4-dimensional Poincar{\'e} conjecture, for $K$ and $K'$ having homeomorphic 0-surgeries $K$ is smoothly slice if and only if $K'$ is smoothly slice.

Early work concerning the Akbulut-Kirby conjecture is due to  Livingston and Kirk-Livingston, who gave examples of knots $K$ for which metabelian invariants of $K \# -K^r$ obstruct the topological concordance of $K$ and and its reverse $K^r$ \cite{Liv83}, \cite{KL99b}. Since knot surgeries are insensitive to the orientation of the underlying knot this provided counterexamples to the original Akbulut-Kirby conjecture
and led to the following revision. 
\begin{conjecture}[(Revised Akbulut-Kirby Conjecture)]
\label{Conj:AKCrev}
If $K$ and $K'$ have $S^3_0(K) \cong S^3_0(K')$ then, up to reversing the orientation of either knot, $K$ and $K'$ are smoothly concordant.  
\end{conjecture}
In 1980 Brakes gave a construction of pairs of knots which share a 0-surgery and yet are distinct as unoriented knots \cite{Bra80}. Using this construction, Gompf-Miyazaki showed the following. 
\begin{proposition}[(\cite{GM95})]
At most one of the slice-ribbon conjecture and Conjecture ~\ref{Conj:AKCrev} is true.
\end{proposition}

The next progress in the resolution of this conjecture came in the work of Cochran-Franklin-Hedden-Horn \cite{CFHH13}, who proved the following in 2013.

\begin{theorem}[\cite{CFHH13}]
There exist knots $K$ and $K'$ whose 0-surgeries are homology cobordant via a cobordism that preserves the homology class of a positive meridian, and yet which are not smoothly concordant. 
\end{theorem}

Finally, Yasui disproved the revised Akbulut-Kirby conjecture with the following result. 

\begin{theorem}[(\cite{Yas15})]\label{thm:Yasui}
There exist infinitely many pairs of knots $K$ and $K'$ with homeomorphic 0-surgeries such that $K$ and $K'$ have different smooth 0-shake genera and so are not smoothly concordant, even up to reversal. 
\end{theorem}

Recall that for $n \in \Z$ the \emph{$n$-trace of a knot} is defined to be the 4-manifold $X_n(K)$ obtained from the four ball by attaching an $n$-framed 2-handle along a neighborhood of $K$. The  \emph{(smooth) $n$-shake genus of $K$} is then defined to be the minimal genus of a smoothly embedded surface that generates the second homology of $X_n(K)$. It is straightforward to verify that $n$-shake genus is both unchanged by reversal and a concordance invariant. The 0-shake genus is by definition an invariant of the 0-trace, so Yasui's examples certainly do not have diffeomorphic 0-traces and the following conjecture of Abe remained open.
\begin{conjecture}[(\cite{Abe16})]\label{Quest:traceconcordance}
If $K$ and $K'$ have diffeomorphic 0-traces then, up to reversing the orientation of either knot, they are smoothly concordant.
\end{conjecture}
The techniques of annulus twisting (see \cite{Oso06}, \cite{AJOT13}) can be used to produce pairs of knots which share a 0-trace. Abe-Tagami consider such a pair and apply a ribbon obstruction due to Miyazaki \cite{Miy94} to show the following.
\begin{proposition}[(\cite{ATag16})]
At most one of the slice-ribbon conjecture and Conjecture ~\ref{Quest:traceconcordance} is true.
\end{proposition}

We extend the techniques of \cite{Bra80} to produce pairs of knots with diffeomorphic 0-traces. We then use the $\di$-invariants of Heegaard Floer homology applied to the double branched cover of our knots in order to disprove Conjecture \ref{Quest:traceconcordance} as follows. 
\begin{thmx}\label{thm:MainTheoremA}
There exist infinitely many pairs of knots  $(K, K')$ such that $K$ and $K'$ have diffeomorphic 0-traces and yet are distinct in smooth concordance, even up to reversal of orientation. 
\end{thmx}

We also exhibit a relationship between the techniques of Brakes and those of annulus twisting in order to use our examples from Theorem \ref{thm:MainTheoremA} to show that there exist knots $K$ and $K'$ related by annulus twisting with $X_0(K) \cong X_0(K')$ but which are not smoothly concordant. \\

In order to discuss several applications of Theorem \ref{thm:MainTheoremA} we recall and reprove a result which has been known to the experts for some time (see \cite{KM78}, \cite{GS99}).

\begin{theorem}{}
$K$ is smoothly slice if and only if $X_0(K)$ smoothly embeds in $S^4$.
\label{Thm:slicetrace}
\end{theorem}
\begin{proof}
For the `only if' direction: Consider $S^4$ and an equatorial $S^3$ therein, which decomposes $S^4$ into the union of two 4-balls $B_1$ and $B_2$. Consider $K$ sitting in this $S^3$. Since $K$ is smoothly slice, we can find a smoothly embedded disk $D_K$ which $K$ bounds in $B_1$. Observe now that $B_2\cup \overline{ \nu (D_K)} \cong X_0(K)$ is smoothly embedded in $S^4$. 

For the `if' direction: Let $F:S^2\to X_0(K)$ be a piecewise linear embedding such that the image of $F$ consists of the union of the cone on $K$ with the core of the two handle. Notice that $F$ is smooth away from the cone point $p$. Let $i:X_0(K)\to S^4$ be a smooth embedding. Then $(i\circ F)$ is a piecewise linear embedding of $S^2$ in $S^4$, which is smooth away from $i(p)$. Note that $W:=S^4\smallsetminus\nu(i(p))\cong B^4$ and that the restriction of $(i \circ F)$ to the complement of a small neighborhood of $F^{-1}(p)$ in $S^2$ is a smooth embedding of $D^2$ in $W \cong B^4$. Further, if we choose this neighborhood to be the inverse image of a sufficiently small neighborhood of $i(p)$ we have that $(i \circ F)(D^2 \smallsetminus \nu(F^{-1}(p)))$ intersects $\partial W$ in the knot $K$. 
\end{proof}
An identical proof shows that $K$ is topologically slice if and only if $X_0(K)$ topologically embeds in $S^4$. 

\begin{corollary}
Let $K$ and $K'$ be knots with $X_0(K)$ diffeomorphic to $X_0(K')$. Then $K$ is smoothly slice if and only if $K'$ is smoothly slice. 
\label{Cor:slicetrace}
\end{corollary}

Using Corollary \ref{Cor:slicetrace} we give a brief proof of the following strengthening of Theorem 3.1 of \cite{ATan16}. Note that we do not require either knot to be ribbon. 
\begin{corollary}
Let $J$ be a knot in $S^3$ admitting an annulus presentation in the sense of \cite{ATan16} and $J_n$ be the knot obtained from $J$ by the $n$-fold annulus twist for some $n \in \Z.$
Then $J$ is smoothly slice if and only if $J_n$ is smoothly slice. 
\end{corollary} 
\begin{proof}
By Theorem 2.8 of \cite{AJOT13}, $X_0(J)$ is diffeomorphic to $X_0(J_n)$.
\end{proof}

Let  $\sim$ denote smooth concordance and $\cc:= \{\text{knots in } S^3\}/ \sim$ denote the smooth concordance group of knots.  Recall that for a pattern $P$ in a solid torus, the operation of taking satellites by $P$ descends to a well defined map $P:\cc\to\cc$.

\begin{definition}\label{defn:invertible}
A pattern $P$ is \emph{(smoothly) invertible}  if there exists a pattern $Q$ such that
$P(Q(K)) \sim K \sim Q(P(K)$ for any $K$. 
\end{definition}
A particularly uninteresting family of invertible patterns is given by those with geometric winding number one. These \emph{connected sum patterns} act by connected sum even on the monoid of knots up to ambient isotopy in $S^3$, and hence certainly descend to invertible maps on $\cc$.  An interesting family of winding number one but not geometric winding one patterns,  called \emph{dualizable} patterns, has been discussed in \cite{Bra80} and \cite{GM95}.   Using Theorem \ref{thm:MainTheoremA} and Corollary \ref{Cor:slicetrace} we give a new proof that dualizable patterns are invertible, a result which first appeared in a stronger form \cite{GM95}. 
\begin{theorem}
Let $P$ be a dualizable pattern, and let $P^{-1}:= \overline{P^*}$ as defined in Section~\ref{subsection:definitions}. Then for any knot $K\subset S^3$,
\[P^{-1}(P(K)) \sim K \sim P(P^{-1}(K)).
\]
\label{Thm:invertible}
\end{theorem}
The work of \cite{DR16} shows that there are dualizable (hence invertible) patterns $P$ whose exteriors are not homology cobordant to the exterior of any connected sum pattern. However, this is not a priori enough to establish that these patterns' actions on the smooth concordance group are distinct from that of connected sum with any knot.  However, by combining Theorem~\ref{Thm:invertible} with Theorem~\ref{thm:MainTheoremA}, we obtain the following. 

\begin{thmx}\label{thm:MainTheoremC}
There are invertible patterns which do not act by connected sum on the smooth concordance group.
\end{thmx}

Note that many of the examples of Theorem~\ref{thm:MainTheoremC} are clearly not automorphisms since they do not have $P(U)\sim U$.
However, as a corollary to Theorem \ref{thm:MainTheoremC} we observe that there exist invertible patterns $P$ with $P(U)$ slice and such that $P(K)$ is not always smoothly concordant to $K$ , see Remark~\ref{remark:puslice}. The actions of these patterns  on $\cc$ are interesting candidates for being group automorphisms of $\cc$ distinct from the identity automorphism. 
It is also perhaps an interesting question whether there exist invertible patterns $P$ which, despite having exteriors distinct from that of a connected sum pattern, still act by connected sum on the knot concordance group.

We say a pair of knots $K$ and $K'$ are $(n,m)$ \emph{0-shake concordant} if there is a smooth properly embedded planar surface $F$ in $S^3\times [0,1]$ with $n+m$ boundary components as follows. The boundary of $F$ has $n=2k+1$ components which are 0-framed copies of $K$ in $S^3\times\{0\}$, $k+1$ of which are parallel to $K$ and $k$ of which are anti-parallel,  and has $m=2j+1$ 0-framed copies of $K'$ in $S^3\times\{1\}$, with analogous orientation requirements. 
Examples of knots which are 0-shake concordant but not concordant first appeared in \cite{CR16}. The pairs we consider in Theorem \ref{thm:MainTheoremA} give new such examples as follows.
\begin{proposition}
\label{Prop:traceimpliesshakeconc}
If $X_0(K)$ is diffeomorphic to $X_0(K')$ then $K$ is (1,r) 0-shake concordant to $K'$ for some $r\in\N$. 
\end{proposition}
\begin{proof}
Let $\phi: X_0(K) \to X_0(K')$ be the diffeomorphism.  
Let $x$ be a point in the interior of $X_0(K)$,  let $x'= \phi(x)$, and let
$X^b_0(K):= X_0(K)\smallsetminus\nu(\{x\})$ and $X^b_0(K'):= X_0(K')\smallsetminus\nu(\{x'\})$.
Then  $\phi$ restricts to a diffeomorphism of $X^b_0(K)$ with $ X^b_0(K')$. 
Let $\widetilde{K}$ be a copy of $K$ in the $S^3$ boundary component of $X^b_0(K)$. There is a smoothly embedded disc $D$ 
in $X^b_0(K)$ with boundary $\widetilde{K}$, which generates  $H_2(X^b_0(K), S^3)$. Therefore $\phi(\widetilde{K})$, a copy of $K$ in the $S^3$ boundary component of $X^b_0(K')$, bounds  $\phi(D)$, a smoothly embedded disk in $X^b_0(K')$ which generates $H_2(X^b_0(K'), S^3)$. 
We can assume without loss of generality that $\phi(D)$ is transverse to the cocore of the 2-handle of $X^b_0(K')$, and so by cutting open along the co-core we obtain a $(1,r)$ 0-shake concordance from $K$ to $K'$. 
\end{proof}

Let $g_4(K)$ denote the \emph{(smooth) 4-genus of} $K$, the minimal genus of any smoothly embedded surface that $K$ bounds in $B^4$. We observe that the annulus twisting construction as used in \cite{AJOT13}, \cite{ATag16}, and \cite{ATan16} produces pairs of knots with diffeomorphic 0-traces which each have 4-genus either 0 or 1.  By Corollary \ref{Cor:slicetrace}, any such pair of knots must have the same 4-genus.
In contrast, our techniques can be used to construct pairs of knots with $X_0(K)\cong X_0(K')$ such that $g_4(K)$ is arbitrarily large. Thus it remains possible that there is some such pair with $g_4(K)\neq g_4(K')$. Since 0-shake genus is by definition an invariant of the 0-trace, such an example would give a negative answer to the following question. 
\begin{question}[(Question 1.41a of \cite{Kir97})]
Must the 4-genus of a knot equal the 0-shake genus of the knot?
\end{question}


We also find compelling the topological analogue of Conjecture \ref{Conj:AKCrev}, given that all known topological concordance invariants are determined by the 0-surgery of a knot.  
\begin{conjecture}
If $K$ and $K'$ have $S^3_0(K) \cong S^3_0(K')$ then, up to reversing the orientation of either knot, $K$ and $K'$ are topologically  concordant.  
\end{conjecture}

Finally, we note that given the Akbulut-Kirby conjecture and its many derivatives it is perhaps surprising that, as far as the authors know, the only known examples of knots with the same 0-surgery that are concordant are in fact slice. 

\begin{question}\label{basicexamples??}
Are there examples of distinct, non-slice knots $K$ and $K'$  such that $S^3_0(K) \cong S^3_0(K')$ and $K$ and $K'$ \emph{are} smoothly concordant?
\end{question}

This paper is organized as follows.  In Section 3 we introduce dualizable patterns and use them to construct pairs of knots with homeomorphic 0-surgeries. We also introduce the examples we will use to prove Theorem~\ref{thm:MainTheoremA}. In Section 4 we show that the pairs from Section 3 have diffeomorphic 0-traces and we prove Theorem~\ref{Thm:invertible}. In Section 5 we use the d-invariants of the double branched covers of the examples from Section 3 to prove Theorems \ref{thm:MainTheoremA} and \ref{thm:MainTheoremC}. In Section 6 we discuss a relationship between annulus twisting and dualizable patterns, and prove that there exist knots related by  annulus twisting which are not smoothly concordant. We conclude in Section 7 by exhibiting an infinite family of knots with interesting surgery properties. 

\section{Acknowledgements}
The first author would like to thank the Hausdorff Institute of Mathematics in Bonn, Germany for providing a stimulating environment, as well as the members of the Knot Concordance and 4-manifolds group there for many interesting and useful conversations. The authors also thank Josh Greene, Min Hoon Kim and Tye Lidman for particularly helpful discussions, and would like to acknowledge the authors of \cite{CHH13}, whose Appendix I we frequently consulted while working on this project. The authors thank their PhD advisors Cameron Gordon and John Luecke for a wealth of advice and encouragement and many insightful conversations. Finally, they thank the anonymous referee for their helpful comments and suggestions. 

\section{A construction of knots with diffeomorphic 0-surgeries}

\subsection{Definitions and notation}\label{subsection:definitions}

Unless specifically mentioned, all maps and concordances are smooth and all knots and manifolds are both smooth and oriented.  We use $\cong$ to denote diffeomorphism of manifolds, $\simeq$ to denote ambient isotopy of knots and links in 3-manifolds, and $\sim$ to denote concordance of knots. All (co)homology groups are by default taken with integer coefficients. For $N$ a properly embedded smooth submanifold of $M$ we use $\nu(N)$ to denote a tubular neighborhood of $N$.


For $K$ a knot in $S^3$ and $p/q$ in the extended rationals we define $S^3_{p/q}(K)$ to be the 3-manifold obtained by $p/q$ Dehn surgery on $S^3$ along $K$. For $n \in \N$ we use $\Sigma_n(K)$ to denote the $n$-fold cyclic cover of $S^3$ branched over $K$ .

Let $P:S^1\to V$ be an oriented knot in a standard solid torus $V:=S^1 \times D^2$.  Assume that $P$ is not null-homologous. By the usual abuse of notation, we use $P$ to refer to both this map and its image.  Define $\lambda_V=S^1\times \{x_0\}$ for some $x_0\in\partial D^2$, oriented so that $P$ is homologous to $n \lambda_V$ for some positive $n$. We call $n$ the \emph{(algebraic) winding number of} $P$. Define $\mu_P$ to be a  meridian for $P$, oriented such that the linking number of $P$ and $\mu_P$ is 1, and define $\mu_V=\{x_1\}\times \partial D^2$ for some $x_1\in S^1$, oriented so that $\mu_V$ is homologous to a positive multiple of $\mu_P$. Finally define the longitude $\lambda_P$ of $P$ to be the unique framing curve of $P$ in $V$ which is homologous to a positive multiple of $\lambda_V$ in $V\smallsetminus \nu(P)$.  Define the \emph{geometric winding number} of $P$ to be the minimal number of intersections of $P$ with the meridional disk for $V$ over all patterns in the isotopy class of $P$. 

 For any knot $K$ in $S^3$ there is a canonical embedding $i_K:V \to S^3$ given by identifying $V$ with $\overline{\nu(K)}$ such that $\lambda_V$ is sent to the null-homologous curve on $\partial\overline{\nu(K)}$. Then $i_K \circ P:S^1\to S^3$ specifies the oriented knot $P(K)$ in $S^3$, the \emph{satellite of $K$ by $P$}. 

Let $\tau_n:S^1\times D^2\to S^1\times D^2$ be the $n$-fold Dehn twist about a positive meridian of $S^1 \times D^2$. Then for a pattern $P$ we define the \emph{$n$-twist} of $P$ to be  $\tau_n(P):= \tau_n \circ P$.

\begin{definition}\label{Defn:dualpattern}
A pattern $P$ is \emph{dualizable} if there exists $P^*:S^1\to S^1\times D^2=:V^*$
such that there is an orientation preserving homeomorphism $f: V\smallsetminus \nu(P)\to V^*\smallsetminus \nu(P^*)$ with $f(\lambda_V)=\lambda_{P^*}, f(\lambda_P)=\lambda_{V^*}, f(\mu_V)=-\mu_{P^*}$, and $f(\mu_{P})=-\mu_{V^*}$. We call $P^*$ the \emph{dual} of $P$. 
\label{Def:dualizable}
\end{definition}

\begin{remark}
\label{rmk:dualconditions}
There is some redundancy in this definition: if one assumes $f(\lambda_P)=\lambda_{V^*}$ and $f(\mu_V)=-\mu_{P^*}$ it is not hard to show that one can isotope $f$ in a small neighborhood of the boundary of $V \smallsetminus \nu(P)$ so that $f(\lambda_V)=\lambda_{P^*}$ and $f(\mu_P)=- \mu_{V^*}$. Thus when checking that a pattern is dualizable or computing a dual it suffices to check that $f(\lambda_P)=\lambda_{V^*}$ and $f(\mu_V)=-\mu_{P^*}$.  We include all four conditions in the definition since we will often appeal to them all when discussing properties of dualizable patterns. 
\end{remark}
We also have the following idea of mirror-reversal for patterns. 
\begin{definition}\label{defn:bar}
Given a pattern $P: S^1 \to V$, define $\overline{P}$ to be the pattern obtained from $P$ by reversing the orientation of both $S^1$ and $V$; note that
$\overline{P}$ has diagram obtained from a diagram of $P$ by changing all crossings and the orientation of $P$. 
\end{definition}
Note that for $P$ dualizable, $\left( \overline{P} \right)^*= \overline{P^*}$. 
We warn the reader that our conventions differ from those in \cite{GM95}, in whose notation our $\overline{P^*}$  is the dual of $P$.

\subsection{Dualizable patterns}
We now describe a method of producing a large class of dualizable patterns and their duals by considering the natural inclusion of $S^1 \times D^2$ into $S^1 \times S^2$ as follows.

\begin{definition}
Define $\Gamma:S^1\times D^2 \to S^1\times S^2$ by $\Gamma(t,d)=(t, \gamma(d))$, where   $\gamma:D^2\to S^2$ is an arbitrary orientation preserving embedding.
 Then for any curve $\alpha: S^1 \to S^1 \times D^2$, we can define a knot in $S^1\times S^2$ by
$\widehat{\alpha} :=\Gamma \circ \alpha:S^1\to S^1\times S^2.$
 \end{definition}

 \begin{proposition}\label{Prop:examples of dualizable patterns}
 A pattern $P$ in a solid torus $V$ is dualizable if and only if $\widehat{P}$ is isotopic to $\widehat{\lambda_V}\}$ in $S^1 \times S^2$. 
 \end{proposition}
\begin{proof}
For the `if' direction, let $V^*= (S^1 \times S^2) \smallsetminus \nu(\widehat{P})$. 
Since $\widehat{P}$ is isotopic to $\widehat{\lambda_V}\simeq S^1 \times \Gamma(\{x_0\})$, where $x_0\in\partial D^2$, we know $V^*$ is a solid torus. Thus there is an identification of $V^*$ with $S^1 \times D^2$ such that $\widehat{\lambda_P}$ is identified with $S^1\times \{pt\}=:\lambda_{V^*}$. 
 Let $Q= \widehat{\lambda_V} \subseteq V^*$. 
Let $Z= (S^1 \times S^2) \smallsetminus \nu(\widehat{P} \sqcup Q)$, and observe that $V \smallsetminus \nu(P) = Z = V^* \smallsetminus \nu(Q)$. Tracing through these identifications, we see that  $\mu_V \leftrightarrow -\mu_Q$ and $\la_P \leftrightarrow \la_{V^*}$ and so by Remark~\ref{rmk:dualconditions} $P$ is dualizable with  $P^*=Q$. 

For the `only if' direction, observe that $S^1\times S^2\smallsetminus\nu(\widehat{P})$ is diffeomorphic to the result of Dehn filling $V\smallsetminus\nu(P)$  along $\mu_V$. Since P is dualizable, this is diffeomorphic to the result of Dehn filling $V^*\smallsetminus\nu(P^*)$ along $-\mu_{P^*}$, which is the solid torus $V^*$. So $\widehat{P}$ is a knot in $S^1\times S^2$ with solid torus complement. By Waldhausen\footnote{In fact, Waldhausen only proves uniqueness up to diffeomorphism. To prove uniqueness up to isotopy requires more work, which is done explicitly in, for example, Carvalho and Oertel \cite{CO05}. For discussion of the history, see \cite{MSZ16} and \cite{Sch07}.} \cite{Wal68} all genus one Heegaard splittings of $S^1\times S^2$ are isotopic, so  $\widehat{P}$ is isotopic to $\pm\widehat{\lambda_V}$.  Since by definition $\widehat{P}$ is homologous to a positive multiple of $\lambda_V$, $\widehat{P}$ must be isotopic to $+\widehat{\lambda_V}$.
\end{proof}
Note that Proposition~\ref{Prop:examples of dualizable patterns}  implies that every geometric winding number one pattern $P$ is dualizable with $P^*= P$. It is also now straightforward to see that all dualizable patterns have (algebraic) winding number $1$, and as a corollary of Proposition~\ref{Prop:examples of dualizable patterns}  we obtain the following fundamental group characterization of dualizability.  (See Theorem~3.4 of \cite{DR16} for a similar result.)
\begin{corollary}\label{cor:pi1condition}
A winding number 1 pattern $P$ in a solid torus $V$ is dualizable if and only if $\mu_P \in \langle \langle \mu_V \rangle \rangle$, the subgroup of $\pi_1(V \smallsetminus P)$ normally generated by $\mu_V$. 
\end{corollary}
\begin{proof}
First, suppose that $\mu_P \in \langle \langle \mu_V \rangle \rangle$ to show that $\widehat{P}$ is isotopic to $\widehat{\la_V}$ in $S^1 \times S^2$. Let $X=(S^1 \times S^2) \smallsetminus \nu (\widehat{P})$.  Note that
$\pi_1(X)= \pi_1(V \smallsetminus P)/ \langle \langle \mu_V \rangle \rangle$ is a quotient of  $\pi_1(V \smallsetminus \nu(P))/ \langle \langle \mu_P \rangle \rangle = \pi_1(V)= \Z$. Since $P$ is winding number one, we have that $P$ is homologous to $\la_V$, and hence that $\widehat{P}$ is homologous to $\widehat{\la_V}$. It follows that $H_1(X) \cong H_1((S^1 \times S^2) - \nu (\widehat{\la_V})) \cong \Z$, and so we must have $\pi_1(X) \cong \Z$ as well. Since $X$ has no $S^2$ boundary components, it follows from \cite{Hem04} that $X$ is homeomorphic to a solid torus. Now the extension of \cite{Wal68} discussed in the proof of Proposition~\ref{Prop:examples of dualizable patterns} shows that $\widehat{P}$ must be isotopic to $\widehat{\la_V}$ in $S^1 \times S^2$, and hence that $P$ is dualizable. 

Now suppose that $P$ is dualizable. It suffices to show that $\mu_{V^*}\in\langle \langle \mu_{P^*} \rangle \rangle$.  However, $\pi_1(V^* \smallsetminus \nu(P^*)) / \langle \langle \mu_{P^*} \rangle \rangle \cong \pi_1(V^*).$
But certainly $[\mu_{V^*}]=0$ in $\pi_1(V^*)$,  and so we have the desired result. 
\end{proof} 

The following example, originally due to \cite{Bra80} and discussed in \cite{GM95} illustrates a method for finding the dual of a dualizable pattern (see also \cite{DR16}).

\begin{example}\label{Exl:Gompf Miyazaki example}
The left of Figure \ref{Fig:examplepattern} shows a pattern $J$ in a solid torus $V$. We will show that $J$ is dualizable with dual $J^* \simeq \tau_{-4}(J)$. Note that $\lambda_V$ is the curve on $\partial V$ which bounds a disk in the diagram. We show $\widehat{J}\subset S^1\times S^2$ (center), where we draw $S^1\times S^2$ as $I\times S^2$ and imagine identifying the points $\partial I\times \{p\}$ for all $p\in S^2$. By isotoping the strand of $\widehat{J}$ which runs `under' for three consecutive crossings around the $S^2$ factor one obtains a diagram of $\widehat{J}$ with those three crossings changed (right). From here it is straightforward to observe that $\widehat{J}$ is isotopic in $S^1\times S^2$ to $S^1\times \{p_0\}$ and so $J$ is dualizable.  We now show that $J^* \simeq \tau_{-4}(J)$. 
\begin{figure}[h]
\includegraphics[width=11cm]{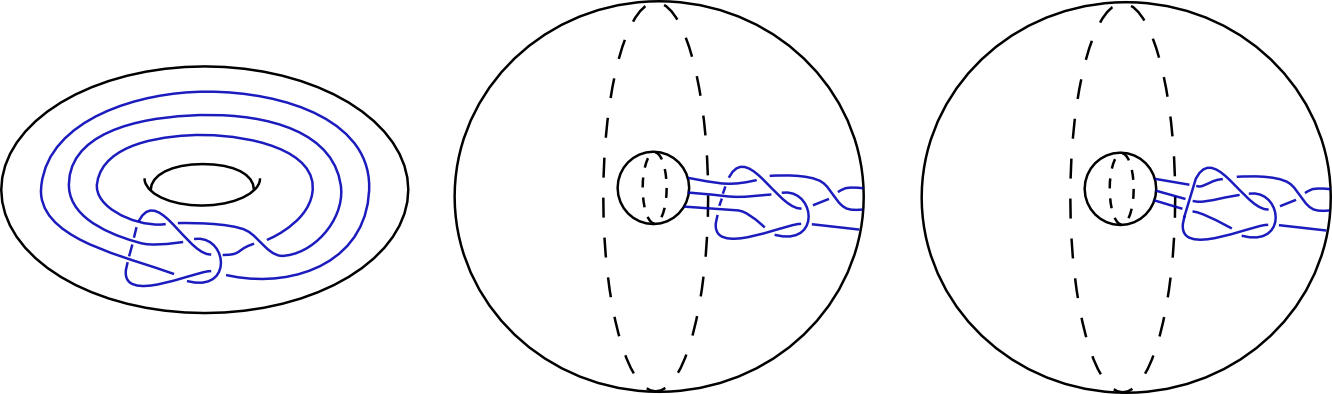}
\caption{On the left, a pattern $J$ in $V$ the standard solid torus in $S^3$. The right two diagrams show that $\widehat{J}$ is isotopic to $S^1\times \{p_o\}$ in $S^1 \times S^2$ }
\label{Fig:examplepattern}
\end{figure}
The top left diagram of Figure \ref{Fig:dualcalculation} shows the link $L:=\widehat{J}\sqcup \widehat{\lambda_V} \subset S^1\times S^2$, where the link is comprised of the red and blue curves. To keep the diagram uncluttered we do not explicitly depict the `outer' $S^2$ in our depiction of $S^1\times S^2$ as a quotient of $I\times S^2$. We also keep track of $\widehat{\lambda_J}$ in green, so we can verify that the conditions of Definition~\ref{Defn:dualpattern} are satisfied.  We isotope as in Figure~\ref{Fig:dualcalculation} so that $\widehat{J}$ goes to $S^1\times \{p_0\}$ as in the bottom left diagram, which we denote $L'$. 
\begin{figure}[h]
\includegraphics[width=13cm]{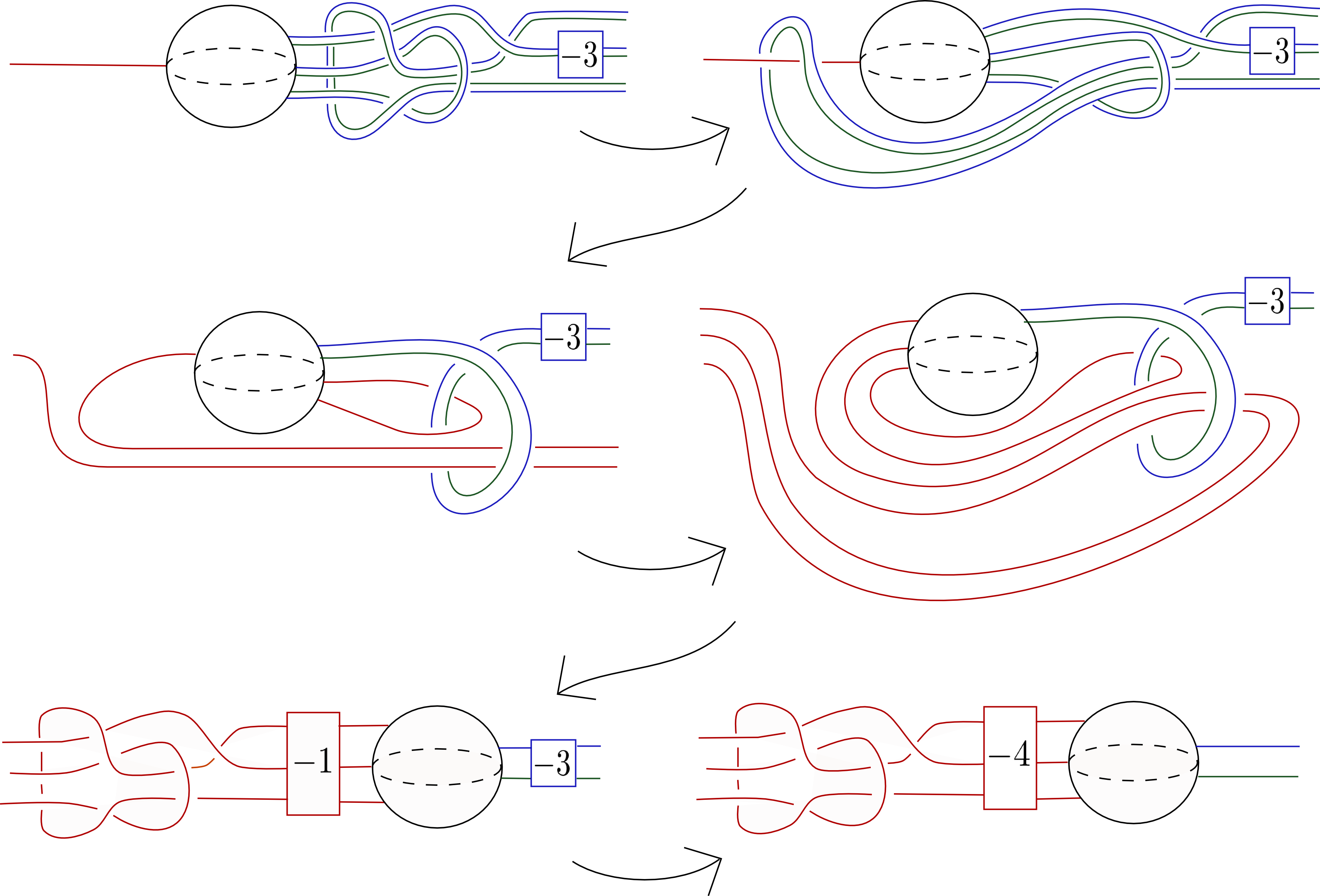}
\caption{Finding $J^* \simeq \tau_{-4}(J)$ by working in $S^1 \times S^2$. }
\label{Fig:dualcalculation}
\end{figure}

Now let $\delta_t:S^2\to S^2$ be the map which rotates $S^2$ by $2\pi t$ about an axis through the east and west poles, and define $\Delta:S^1\times S^2\to S^1\times S^2$ by $\Delta(t,d)=(t,\delta_t(d))$ for all $t\in S^1$, $d\in S^2$. Observe that $\Delta^3$ is a self-homeomorphism of $S^1 \times S^2$ which takes $L'$ to the link $L''= (S^1 \times \{x_0\}) \sqcup \eta$  shown in the bottom right of Figure \ref{Fig:dualcalculation}. 

The composition of $\Delta^3$ with the homeomorphism induced by the isotopy of $L$ to $L'$ gives a self-homeomorphism of $S^1\times S^2$ which takes $L$ to $L''$, and hence a homeomorphism $f:(S^1\times S^2)\smallsetminus \nu(L)\to (S^1\times S^2)\smallsetminus \nu(L'')$. Note that $(S^1\times S^2)\smallsetminus \nu(L)= V\smallsetminus \nu(J)$. Identify $(S^1 \times S^2) \smallsetminus (S^1 \times \{x_0\})$ as a solid torus $V^*$, where the identification is made so that $f(\lambda_J)= \lambda_{V^*}$. Finally, we can  identify $(S^1\times S^2)\smallsetminus \nu(L'')$ as $V^*\smallsetminus \nu(J^*)$ for some $J^* \hookrightarrow V^*$ so that $f$ gives a homeomorphism from $V \smallsetminus \nu(J)$ to $V^* \smallsetminus \nu(J^*)$ with the hypotheses of Definition \ref{Defn:dualpattern}  satisfied. 
\label{Ex:J}
\end{example}

\begin{theorem}[(\cite{Bra80})]
If $P$ is a dualizable pattern with dual $P^*$, then there is a homeomorphism $\phi:S^3_0(P(U))\to S^3_0(P^*(U))$. 
\label{Thm:diffeosurgeries}
\end{theorem}

\begin{proof}
Construct $S^3_0(P(U))$ by Dehn filling $V\smallsetminus\nu(P)$ along $\lambda_P$ and $\lambda_V$. Since $P$ is dualizable, this is diffeomorphic to $V^*\smallsetminus\nu( P^*)$ Dehn filled along $\lambda_{V^*}$ and $\lambda_{P^*}$, which is $S^3_0(P^*(U))$. 
\end{proof}

\begin{proposition}\label{Prop:twist}
If $P$ is dualizable then so is $\tau_n(P)= \tau_n \circ P$, with dual pattern $(\tau_nP)^*=\tau_{-n}(P^*)$. 
\end{proposition}

\begin{proof}
We construct $(\tau_n P)^*$. For this argument we will draw $P$ as the closure in $S^1\times D^2$ of a $2r+1$ strand tangle which we will denote with a box labeled $P$. Any parallel within this tangle box, as in Figure~\ref{Fig:dualtwist}, will be taken according to $\lambda_P$. 

In Figure \ref{Fig:dualtwist} we start to construct $(\tau_nP)^*$ as in Example \ref{Ex:J}.  In order to construct a homeomorphism between $V\smallsetminus \nu(P)$ and the complement of some pattern in a solid torus $V^*$ we look for a self-homeomorphism of $S^1\times S^2$ which takes $\widehat{\tau_n(P)}\sqcup\widehat{\lambda_V}$ to a link $L'$ such that $\widehat{\tau_n(P)}$ is sent to an $S^1\times \{pt\}$ component of $L'$. We also keep track of a copy of $\widehat{\lambda_{\tau_n(P)}}$ in green; note that within the tangle box the green curve is $\lambda_P$ framing the blue curve $P$. 

As a first step we apply $\Delta^{-n}$ to remove the $n$ twists from $\widehat{\tau_n(P)}$. Then the resulting link is $\widehat{P}\sqcup\widehat{\lambda_V}$, and the green curve is a copy of $\Gamma(\lambda_P-n\mu_P)$. Then since $P$ is dualizable there is a homeomorphism to the third diagram in Figure \ref{Fig:dualtwist}. Applying $\Delta^n$ we get a link $L'$ in which the images of $\widehat{\tau_n(P)}$ and $\widehat{\lambda_{\tau_n(P)}}$ are unlinked and isotopic to $S^1\times \{p_0\}$. From this diagram we can read off $(\tau_n P)^*$ and we see that it is $\tau_{-n}(P^*)$. 
\begin{figure}[h]
\includegraphics[width=8cm]{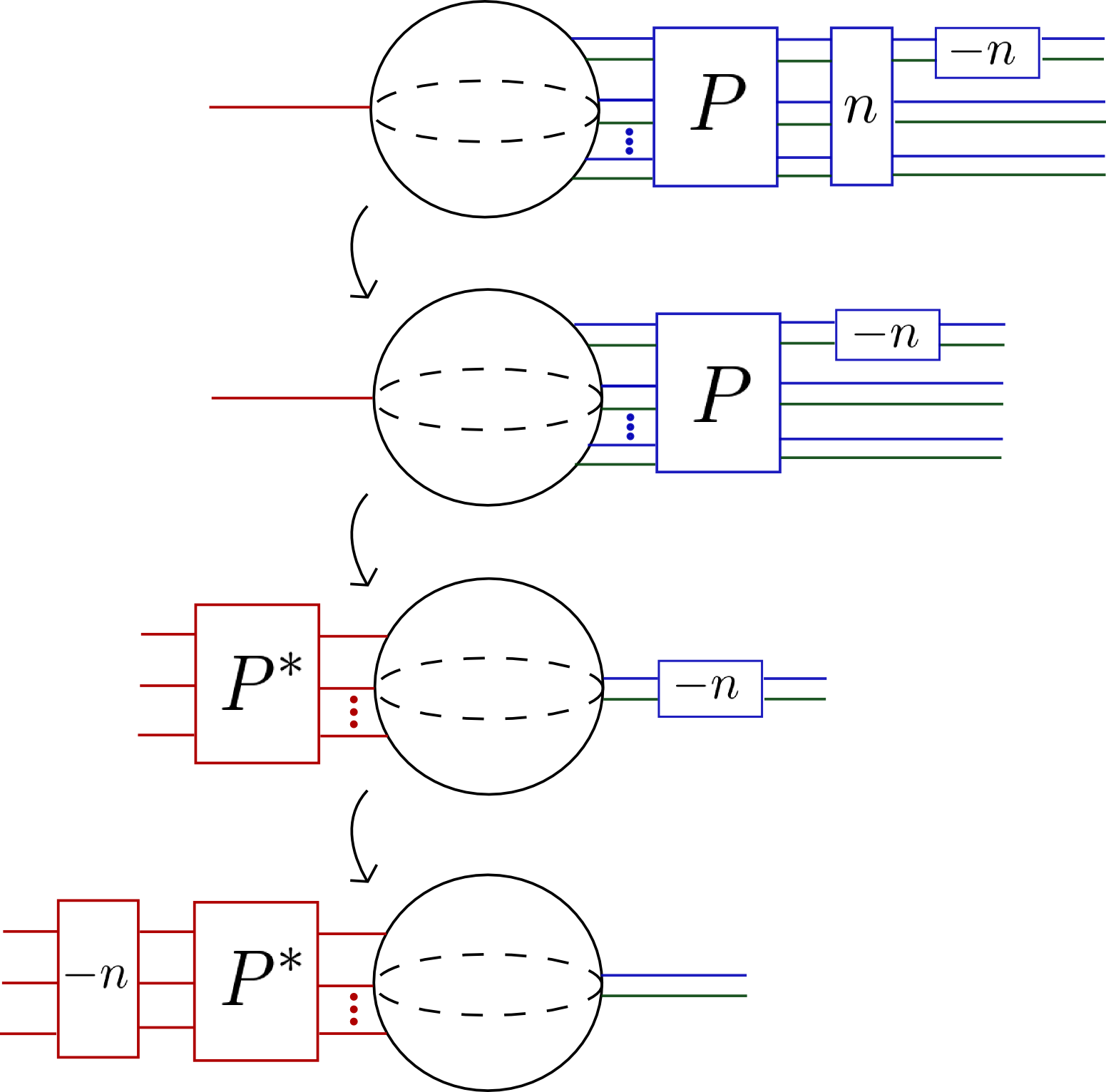}
\caption{The dual of $\tau_n P$ is $\tau_{-n} P^*$.}
\label{Fig:dualtwist}
\end{figure}
\end{proof}


\begin{proposition}
Let $P$ be a dualizable pattern with dual $P^*$. 
Then for any $n \in \Z$, $S^3_n(P(U)) \cong S^3_n\left[(\tau_{n}P^*)(U)\right]$. 
\end{proposition}

\begin{proof}
It is straightforward to generalize the proof of Theorem~\ref{Thm:diffeosurgeries} to this setting. 
\end{proof}

\section{Extending the 0-surgery diffeomorphism across the 0-trace}
Our goal of this section will be to prove the following theorem.
\begin{theorem}
Let $P$ be a dualizable pattern. Then there is a diffeomorphism $\Phi:X_0(P(U))\to X_0(P^*(U))$. 
\label{Thm:diffeotraces}
\end{theorem}

First, we recall a result of Akbulut which first appeared in \cite{Akb77}. We refer the reader to proofs in \cite{AJOT13} and \cite{Akb16}. For the details of handle calculus see \cite{GS99}.
\begin{lemma}
Let $M$ and $N$ be  4-manifolds with a homeomorphism $\psi:\partial M\to\partial N$. If the following are true of $\psi$, then there exists a diffeomorphism $\Psi:M\to N$ such that $\Psi|_{\partial}=\psi$.
\begin{enumerate}
\item There exists some $K:S^1\to\partial  M$ which bounds a smoothly embedded disk $D_K$ in $M$ and with the property that $\psi(K)$ bounds a smoothly embedded disk $D_{\psi(K)}$ in $N$.
\item Let $D_K'$ be a section of $\overline{\nu(D_K)}$ and $D_{\psi(K)}'$ be a section of $\overline{\nu(D_{\psi(K)})}$. Then $\psi(\partial D_K')$ and $\partial D_{\psi(K)}'$ induce the same framing on $\psi(K)$.
\item $M\smallsetminus\nu(D_K)\cong N\smallsetminus\nu(D_{\psi(K)})\cong B^4$.
\end{enumerate}
\label{Lem:akbulutextension}
\end{lemma}

\begin{proof}[(Theorem \ref{Thm:diffeotraces})]
We will check that  for $M:=X_0(P(U))$ and $N:=X_0(P^*(U))$ the homeomorphism $\phi:\partial M\to \partial N$ as in Theorem \ref{Thm:diffeosurgeries} satisfies the conditions of Lemma~\ref{Lem:akbulutextension}.

 Let $K$ be $\mu_{P(U)}= i_U(\mu_P)$ in $\partial M$. Then $K$ bounds a disk $D_K$  in $M$, namely a disk in $S^3$  with its interior pushed slightly into $B^4 \subseteq X_0(P(U))=M$. One checks that $\partial D'_K$ induces the 0-framing on $K$. By inspection of $\phi$ we see that $\phi(K)$ is $i_U(\mu_{V^*})\subset \partial N$. Just as we saw for $K$, the knot $\phi(K)$ bounds a disk $D'_{\phi(K)}$ in $N$ and $\partial D'_{\phi(K)}$ induces the 0-framing on $\phi(K)$.  Further inspection of $\phi$ shows that $\phi(\partial D'_K)$ induces the 0-framing on $\phi(K)$. So conditions (1) and (2) are satisfied. 

Figure \ref{Fig:traceshandle} demonstrates handle diagrams of  $M\smallsetminus\nu(D_K)$ and $N\smallsetminus\nu(D_{\phi(K)})$. It is clear that $M\smallsetminus\nu(D_K)\cong B^4$. 
\begin{figure}[h]
\includegraphics[width=10cm]{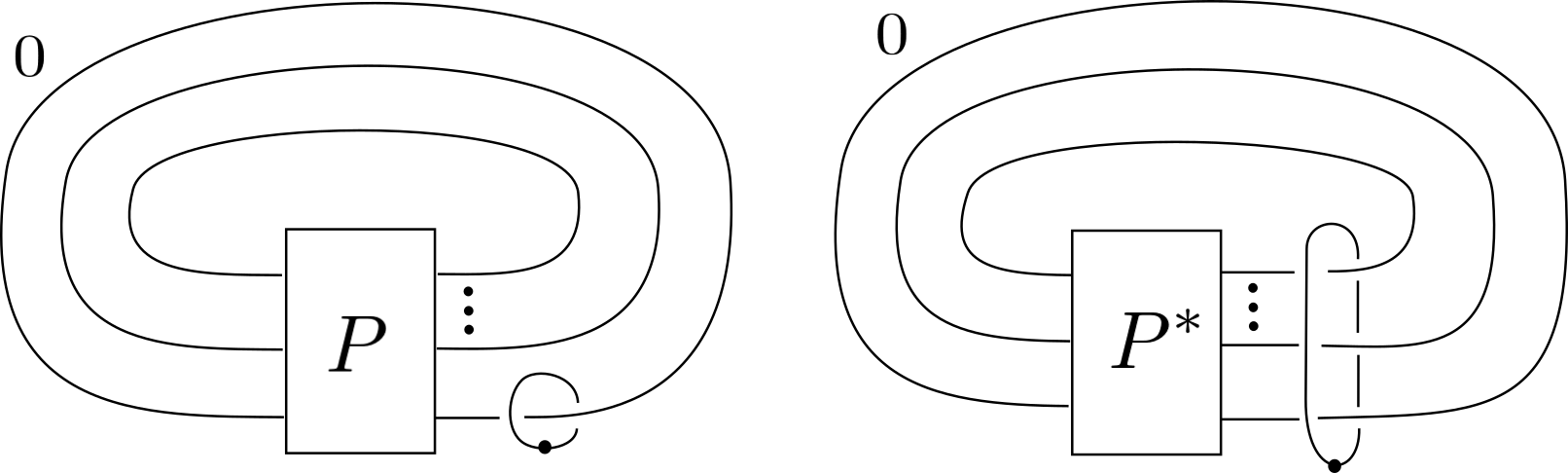}
\caption{Handle diagrams for $M\smallsetminus\nu(D_K)$ and $N\smallsetminus\nu(D_{\phi(K)})$, respectively.}
\label{Fig:traceshandle}
\end{figure}

Observe that $\partial (N\smallsetminus\nu(D_{\phi(K)}))$ can be interpreted as a Dehn surgery on $\widehat{P^*}\subset S^1\times S^2$. By Proposition \ref{Prop:examples of dualizable patterns}, there is some isotopy of $\widehat{P^*}$ to $S^1\times\{pt\}$.
This implies that in the given handle decomposition of $N\smallsetminus\nu(D_{\phi(K)})$ there is some sequence of slides of the two handle over the one handle which results in a handle diagram for $N\smallsetminus\nu(D_{\phi(K)})$ as in Figure \ref{Fig:ntraceslide}. It is clear from Figure \ref{Fig:ntraceslide} that $N\smallsetminus\nu(D_{\phi(K)})\cong B^4$. 

\begin{figure}[h]
\includegraphics[width=3cm]{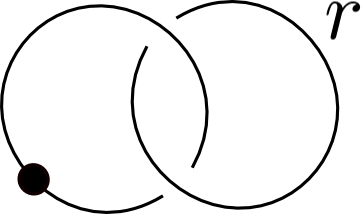}
\caption{ A handle diagram for $N\smallsetminus\nu(D_{\phi(K)})$, where $r$ is some integer.}
\label{Fig:ntraceslide}
\end{figure}
\end{proof}

For a pattern $Q$ in a solid torus $V$, let $j_Q:S^1\times D^2 \to V$ be an identification of $S^1 \times D^2$ with $\overline{\nu(Q)}$ such that $j_Q(S^1\times\{pt\})=\lambda_Q$. For $P$ a pattern in $S^1 \times D^2$ we define $P\circ Q:=j_Q\circ P:S^1\to V$. 

\begin{proposition}
\label{Prop:compositedual}
Let $P$ and $Q$ be dualizable patterns. Then $P\circ Q$ is dualizable with dual $Q^*\circ P^*$. 
\end{proposition}
\begin{proof}
For clarity, we write each pattern $R \in \{ P, Q, P \circ Q, Q^* \circ P^*\}$ as living in a solid torus $V_R$. 
Observe that $V_{P \circ Q} \smallsetminus\nu(P\circ Q)$ can be written as the union of $(V_Q\smallsetminus\nu(Q))$ with $(V_P\smallsetminus\nu(P))$ via an identification of  $\lambda_{V_P}$ with $\lambda_Q$ and $\mu_{V_P}$ with $\mu_Q$. Since $P$ and $Q$  are dualizable, we see that $V_{P \circ Q}\smallsetminus\nu(P\circ Q)\cong (V_Q^*\smallsetminus\nu(Q^*))\cup(V_P^*\smallsetminus\nu(P^*))$ where the gluing  identifies $\lambda_{P^*}$ with $\lambda_{V_Q^*}$ and $\mu_{P^*}$ with $-\mu_{V_Q^*}$. But this gives a description of $V_{Q^* \circ P^*}\smallsetminus \nu(Q^*\circ P^*)$ in which it is straightforward to check that the longitude and meridian conditions of Definition~\ref{Defn:dualpattern} are satisfied. 
\end{proof}

\begin{corollary}\label{Cor:sumtrace}
For $P$ a dualizable pattern and $K$ a knot in $S^3$, $X_0(P(K))\cong X_0(P^*(U) \# K).$
\end{corollary}
\begin{proof}
Let $K_\#$ be the geometric winding number one pattern with $K_\#(U)\simeq K$. Observe that $K_\#$ is dualizable and self dual. Then consider $(P\circ K_\#)(U)$ and apply Proposition~\ref{Prop:compositedual} and Theorem~\ref{Thm:diffeotraces}. 
\end{proof}

\begin{corollary}
For $P$ dualizable, $\overline{P^*}(P(U))\sim U\sim P(\overline{P^*}(U))$.
\label{Cor:invertibleforU}
\end{corollary}
\begin{proof}
Let $P_\#$ be the geometric winding number one pattern with $P_\#(U)=P(U)$. Then by Theorem~\ref{Thm:diffeotraces} and Proposition~\ref{Prop:compositedual} we have
 \[X_0(\overline{P^*}(P(U)))\cong X_0(\overline{P^*}(P_\#(U)))\cong X_0((\overline{P^*}\circ P_\#)(U))\cong X_0((P_\#\circ \overline{P})(U)))\cong X_0(P(U)\# \overline{P}(U))\]
It follows from Definition~\ref{defn:bar} that $P(U)\# \overline{P}(U)$ is slice, so by Corollary~\ref{Cor:slicetrace}, we have that  $\overline{P^*}(P(U))$ is slice as well. The other concordance follows similarly. 
\end{proof}

\begin{proof}[(Theorem~\ref{Thm:invertible})]
Let $K_\#$ be the geometric winding number one pattern with $K_\#(U)=K$ as above. Observe that 
\[\overline{P^*}(P(K))\# -K = (\overline{K_\#}\circ \overline{P^*})( P\circ K_\# )(U) =\overline{(K_\#\circ P^*)}( P\circ K_\# )(U) .\]
 By Proposition \ref{Prop:compositedual}, we have that $(K_{\#} \circ P^*)$ is the dual of $(P \circ K_{\#})$.  
 Thus by Corollary \ref{Cor:invertibleforU} $\overline{P^*}(P(K))\#-K$ is slice and $\overline{P^*}(P(K))\sim K$.
 The argument that $P(\overline{P^*}(K))\sim K$ is analogous. 
\end{proof}

The following generalization of Theorem~\ref{Thm:diffeotraces} is proved via an identical argument; we leave the proof to the reader.
\begin{theorem}
\label{Thm:diffeontraces}
For $P$ a dualizable pattern, $X_n(P(U)) \cong X_n\left[(\tau_{n}P^*)(U)\right]$.
\end{theorem}

\section{Distinguishing certain $P(U)$ and $P^*(U)$ in concordance}

We consider the double branched covers of knots $K_k:=(\tau_{2k-1} J)(U)$ for $k\in\Z$ and $J$ the dualizable pattern of Example~\ref{Ex:J}. A standard argument shows that if two knots $K$ and $K'$ are  concordant then their double branched covers $\Sigma_2(K)$ and $\Sigma_2(K')$ are rational homology cobordant. We will use the d-invariants of \Ozsvath and \Szabo to show that there is no such cobordism between the double branched covers of certain $K_k$ and $K_j$. For the reader's convenience, we state the facts we will need about the behavior of d-invariants.

\begin{theorem}[(Theorem 9.6 of \cite{OS03abs})] \label{thm:dinvariantfacts}
To a rational homology sphere $Y^3$ and a $\Spc$-structure $\s$ on $Y$ there is an invariant $d(Y, \s) \in \Q$ satisfying the following properties:
 If $W$ is a cobordism from $Y_0$ to $Y_1$ (that is, with $\partial W= -Y_0 \sqcup Y_1$), then for any $\s \in \Spc(W)$
\begin{align}
 \di(Y_1, \s|_{Y_1}) &\geq \di(Y_0, \s|_{Y_0}) + \frac{c_1(\s)^2 + \beta_2(W)}{4} \text{ if } W \text{ is negative definite.}
  \label{fact:+def} \\
 \di(Y_1, \s|_{Y_1}) &\leq \di(Y_0, \s|_{Y_0}) + \frac{c_1(\s)^2 - \beta_2(W)}{4} \text{ if } W \text{ is positive definite.}
   \label{fact:-def}\\
 \di(Y_1, \s|_{Y_1}) &= \di(Y_0, \s|_{Y_0}) \text{ if } W \text{ is a rational homology cobordism.} \label{fact:rationalhomologycob}
\end{align}
\end{theorem}

Note that $\Spc(Y)$ can be non-canonically put in bijective correspondence with $H^2(Y)$ and so an integer homology sphere $Y$ has a single $\Spc$-structure and hence a single $\di$-invariant, which we refer to as $\di(Y)$. 

\begin{proposition}[(Corollary 1.5 of \cite{OS03alt})] \label{prop:alternatingsurgery}
Let $K$ be an alternating knot. Then 
\[\di(S^3_1(K))= 2 \min \left\{0,  -\left\lceil\frac{-\sigma(K)}{4}\right\rceil\right\}
\]
\end{proposition}

Now observe that $K_k$, shown on the left of Figure~\ref{Fig:surgerydescriptionofk}, has a surgery description as on the right of Figure~\ref{Fig:surgerydescriptionofk}. Let $\eta$ be the red $(-\frac{1}{2k})$- framed curve and $\gamma$ be the blue $(+1)$-framed curve. Some isotoping gives us the surgery description on the left of Figure~\ref{Fig:dbcsurgeryk}. 
From this we obtain the surgery description for $Y_k:= \Sigma_2(K_k)$ on the right of Figure~\ref{Fig:dbcsurgeryk}. For the reader's convenience we say a few words justifying the surgery coefficients, where for any curve $\sigma$ in a diagram we use $\la_{\sigma}^{bb}$ to refer to the blackboard-framed longitude of $\sigma$. 
\begin{figure}[h]
\includegraphics[height=5cm]{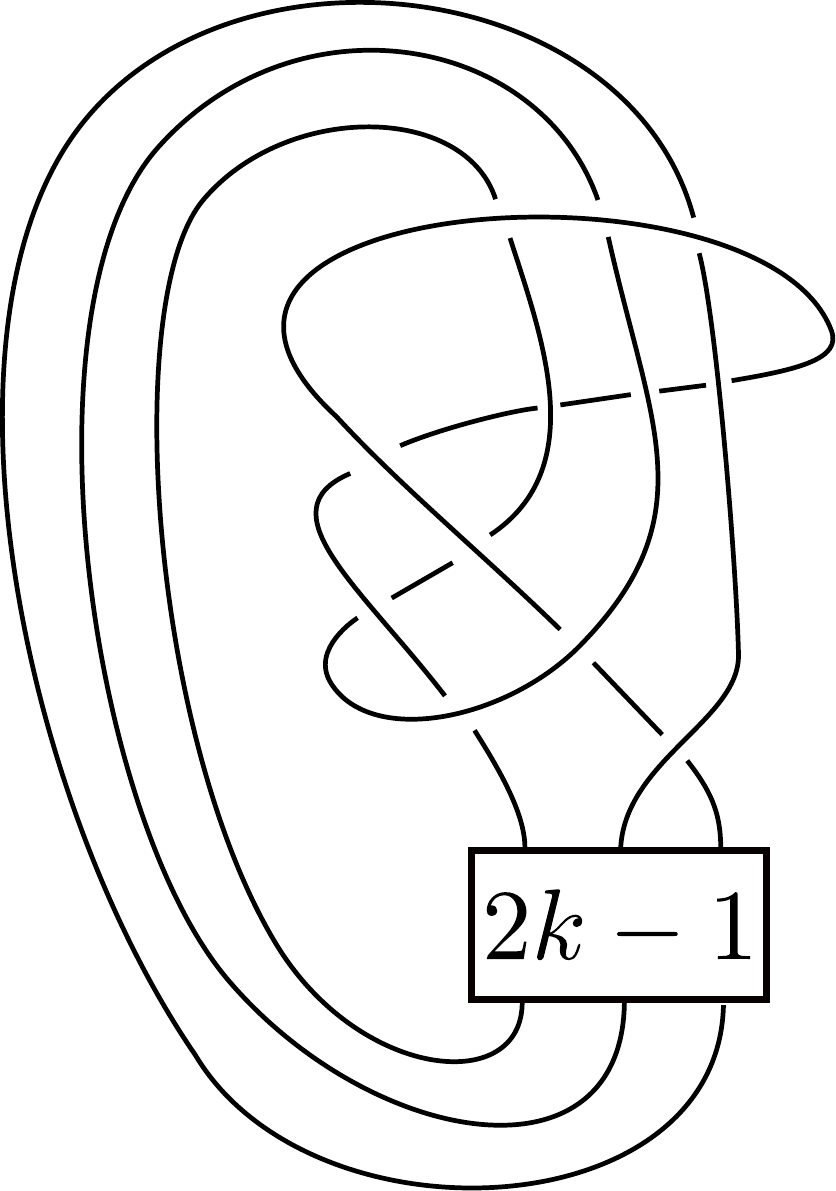} \quad \quad \includegraphics[height=5cm]{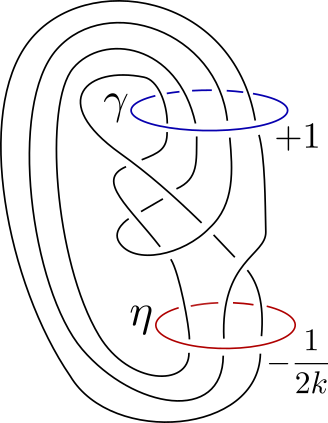}
\caption{$K_k$ (left)  has an unknotted surgery description (right).}
\label{Fig:surgerydescriptionofk}
\end{figure}
\begin{figure}[h]
\begin{tabular}{@{}c@{}}{\includegraphics[height=6cm]{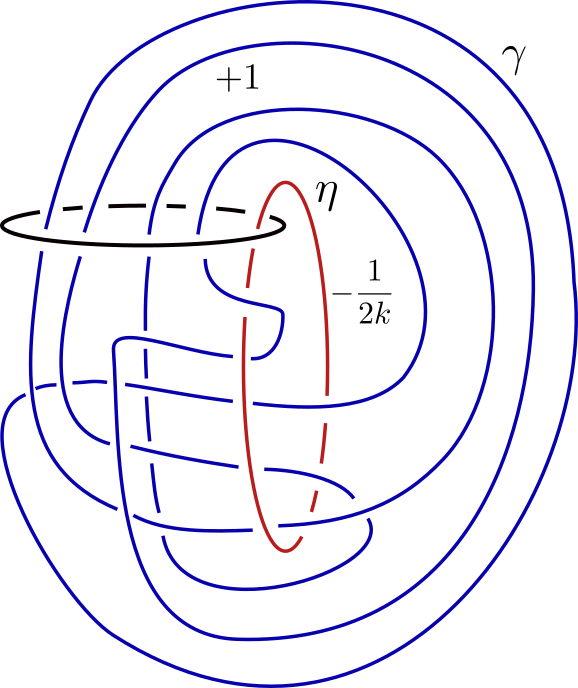}}\end{tabular}
\quad  \quad  \quad
\begin{tabular}{@{}c@{}}{\includegraphics[height=7cm]{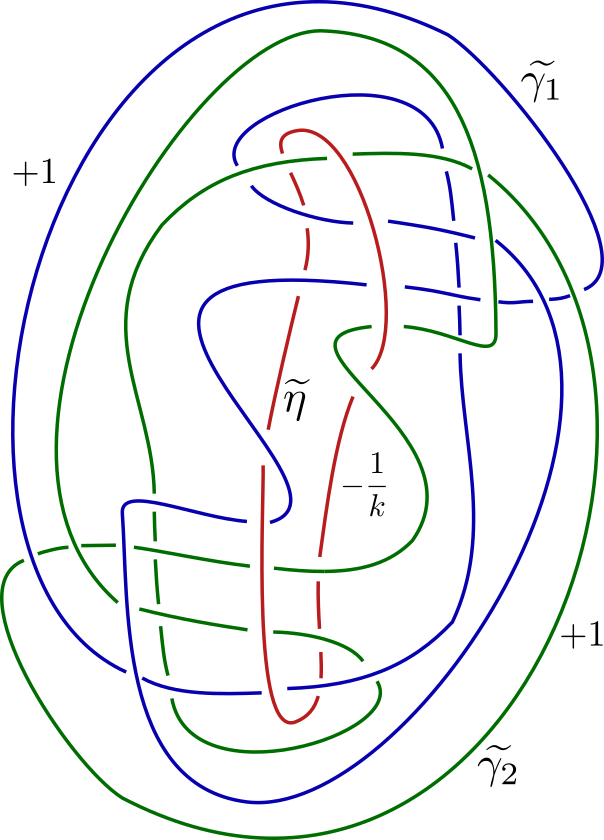}}\end{tabular}
\caption{$K_k$ now appears as the standard unknot (left), which gives a surgery description for $Y_k= \Sigma_2(K_k)$ (right).}
\label{Fig:dbcsurgeryk}
\end{figure}
We see on the left side of Figure~\ref{Fig:dbcsurgeryk} that $\eta$ has writhe 0, so we can write $\fr_{\e}= -\mu_{\e}+ 2k\la_{\e}^{bb}$. 
The preimage of $\mu_\e$ under the branched covering map is $\mu_{\widetilde{\e}}$ and the preimage of $2 \la_{\e}^{bb}$ is $\la_{\widetilde{\e}}^{bb}$. Since $\widetilde{\eta}$ also has writhe 0 the preimage of $\fr_{\e}$ is 
\begin{align*}
\fr_{\widetilde{\e}}=  -\mu_{\widetilde{\e}}+k\la_{\widetilde{\e}}^{bb}= -\mu_{\widetilde{\e}}+k\la_{\widetilde{\e}} 
\end{align*}
so we have a surgery coefficient of $-1/k$ for $\widetilde{\e}$. 
Similarly, on the left of Figure~\ref{Fig:dbcsurgeryk} we see that  $\g$ has writhe $+2$. So $\fr_{\g}= \mu_{\g} + \la_{\g}= -\mu_{\g}+ \la_{\g}^{bb}$.
Since $\gamma$ lifts to two closed curves $\widetilde{\g_1}$ and $\widetilde{\g_2}$, the preimages of $\mu_\g$ are $\mu_{\widetilde{\g_i}}$ and of $\la_\g^{bb}$ are $\la_{\widetilde{\g_i}}^{bb}$ for $i=1,2$. Since the $\widetilde{\g_i}$ both have writhe $+2$ we have
\begin{align*}
\fr_{\widetilde{\g_i}}= -\mu_{\widetilde{\g_i}}+ \la_{\widetilde{\g_i}}^{bb}= -\mu_{\widetilde{\g_i}}+ (\la_{\widetilde{\g_i}} + 2\mu_{\widetilde{\g_i}})= \mu_{\widetilde{\g_i}}+ \la_{\widetilde{\g_i}} \text{ for } i=1,2
\end{align*}
so we have surgery coefficients of $+1$ for the $\widetilde{\g_i}$.

\begin{proposition}\label{prop:monotonicity}
Let $K_k$ and $Y_k$ be as above. Then $Y_k$ is an integer homology sphere and hence has a single d-invariant $\di(Y_k)$ satisfying
\[\di(Y_{-k}) \leq \di(Y_0) \leq \di(Y_{k}) \text{ for } k \in \N.\] 
\end{proposition}

\begin{proof}
From Figure~\ref{Fig:dbcsurgeryk}, we see that $Y_0$ has a surgery description $\widetilde{\g_1} \cup \widetilde{\g_2}$ whose linking matrix is 
$\left[ 
\begin{array}{cc}
1&0 \\
0&1
\end{array}
\right]
$ and so $Y_0$ is a homology sphere. Also, $Y_k$ is obtained from $Y_0$ by $(-1/k)$-surgery on $\widetilde{\e}$, which has zero algebraic linking with $\g_1$ and $\g_2$, and so $Y_k$ is also a homology sphere. 
Now recall that $k>0$, and observe that $-1/k$
 has continued fraction expansion $[-1,-2, -2, \dots, -2]$, where there are $(k-1)$ occurrences of `$-2$'. Therefore we have a surgery diagram for $Y_k$ that differs from that of Figure~\ref{Fig:dbcsurgeryk} only in a small neighborhood of a point on $\widetilde{\e}$ as indicated in Figure~\ref{Fig:rationaltointeger} (see \cite{GS99} for justification). 
\begin{figure}[h]
\includegraphics[height=2.5cm]{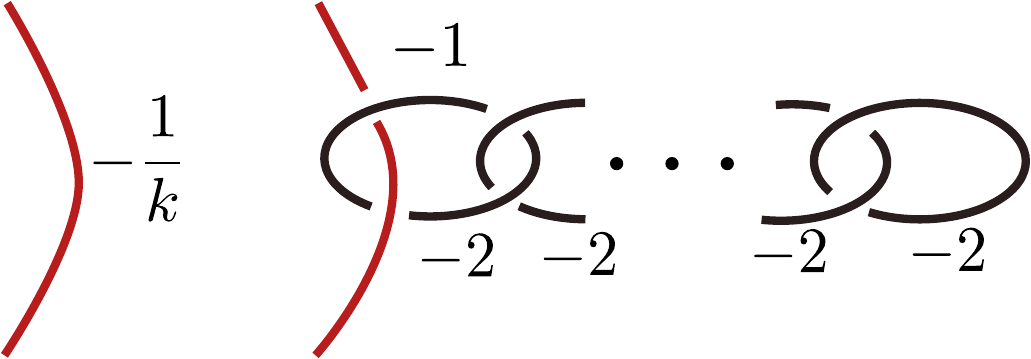}
\caption{An integer surgery diagram for $Y_k$.}
\label{Fig:rationaltointeger}
\end{figure}
Let $W_k$ be the 4-manifold obtained from $Y_0 \times I$ by attaching $k$ 2-handles, one with $(-1)$- framing and $(k-1)$ with $(-2)$-framing as indicated in Figure~\ref{Fig:rationaltointeger}. Notice that $W_k$ is a cobordism from $Y_0$ to $Y_k$ and by standard calculations
 \[ H^{j}(W_k) \cong H_j(W_k)
 \cong
\left\{\begin{array}{cc} 
 0 & j=1\\
 \Z^k & j=2\\
 \Z & j=3
 \end{array}
 \right. . 
\]

Let $S_i$ be the union of the core of the $i$th 2-handle of $W_k$ with a null-homology of the attaching circle of this handle in $Y_0$. Notice that $S_1, \dots, S_k$ generate $H_2(W_k)$. With respect to this basis the intersection form of $W_k$ is given by the $(k \times k)$ linking matrix of the attaching circles, which is
\[
Q_k = \left[
\begin{array}{cccccc}
-1 & -1 &   0 & 0&  \cdots &0 \\
-1 & -2 & -1 &0 &\cdots &0 \\
0 & -1 & -2 & -1 & \cdots & 0 \\
\vdots & \ddots & \ddots & \ddots & \ddots &\vdots\\
0 & \cdots &  -1 & -2 & -1 &0\\
0 & \cdots &  0 & -1 & -2 &-1\\
0 & \cdots &  0 &0& -1 & -2 
\end{array}
\right]. 
\]
It is straightforward to show inductively that given $v= \left[v_1 \, \dots \, v_k\right]$ we have 
\begin{align*}\label{eqn:intform}
v Q_k v^{T}=-(v_1 + v_2)^2 - (v_2+ v_3)^2 -\dotsc-(v_{k-1}+ v_k)^2 - v_k^2.
\end{align*}
This gives an isomorphism from $Q_k$ to the standard negative-definite intersection form on $\Z^k$. If $\s$ is any $\Spc$ structure on $W_k$, then $\s$ must restrict on $\partial W_k$ to the unique $\Spc$ structures on $Y_0$ and $Y_k$, and Theorem~\ref{thm:dinvariantfacts} tells us that 
\[ \di(Y_k) \geq \di(Y_0)+ \frac{ c_1(\s)^2 + \beta_2(W_k)}{4} 
.\]

The map $c_1: \Spc(W_k) \to H^2(W_k)$ has image consisting exactly of all the characteristic vectors. That is, letting $\cdot$ denote the intersection form on second cohomology, the characteristic classes are those elements $v$ with $v \cdot w\equiv w \cdot w \mod 2$ for all $w \in H^2(W_k)$. So in order to finish the proof it is enough to show that there is some characteristic vector $\xi \in H^2(W_k) \cong \Z^k$ with $\xi \cdot \xi + k \geq 0$. But this follows immediately from our isomorphism to the standard negative-definite intersection form, since with respect to the basis for $H_2(W_k)$ given by $\{v_1+v_2, v_2 + v_3, \dots, v_{k-1}+v_{k}, v_k\}$ the vector $\xi= [1, 1, \dots, 1, 1]$ is easily seen to be characteristic and satisfy $\xi \cdot \xi=-k$.

The argument to show $\di(Y_{-k}) \leq \di(Y_0)$ is almost identical: observe that $Y_{-k}$ is obtained from $Y_0$ by $1/k$ surgery on $\widetilde{\e}$,  note that $1/k$ has continued fraction expansion $[1, 2, \dotsc, 2]$ (with $k-1$ occurences of `2'), build a cobordism $W_k^+$ from $Y_0$ to $Y_k$ with $k$ 2-handles, note that $W_k^{+}$ is positive-definite, and apply Theorem~9.6 of \cite{OS03abs}. 
\end{proof}

In fact, the proof of Proposition~\ref{prop:monotonicity} extends to show the following. 
\begin{proposition}\label{prop:generalmonotonicity}
Let $L$ be a knot in $S^3$ and $\eta$ be an unknot which has linking number $\pm1$ with $L$. 
Suppose that $\Sigma_n(L)$ is an integer homology sphere for some $n \in \N$. 
For $k \in \Z$, let $\tau_{nk}(L,\eta)$ denote the $nk$-fold Rolfsen twist of $L$ along $\eta$. 
Then for any $k \in \N$,
\[ \di(\Sigma_n(\tau_{-nk}(L, \eta))) \leq \di(\Sigma_n(L)) \leq \di(\Sigma_n(\tau_{nk}(L, \eta))).\]
\end{proposition}
\begin{proof}
 The argument goes as follows: unknot $L$ by $(\pm 1)$-surgery on unknotted curves $\{ \gamma_i\}$ which algebraically link each of $L$ and $\eta$ zero times to construct a surgery diagram for $\tau_{nk}(L,\eta)$. (One can always do this by choosing small surgery curves about unknotting crossings for $L$.) Adding the curve $\eta$ with framing $(-1/nk)$ gives a surgery diagram for $\tau_{nk}(L,\eta)$. Then follow the above algorithm to obtain a surgery diagram for $\Sigma_n(\tau_{nk}(L,\eta))$. It suffices to show that $\eta$ lifts to a single null-homologous curve with framing $(-1/k)$ in the surgery diagram for $\Sigma_n(L)$ in order to proceed exactly as in the proof of Proposition~\ref{prop:monotonicity}. 
 
One checks that $\eta$ lifts to a curve with framing $-1/k$ exactly as we lifted coefficients above. Since $\eta$ has linking number one with $L$, it has preimage a single curve $\widetilde{\eta}$ in our surgery diagram for $\Sigma_n(\tau_{nk}L)$. Since $\eta$ has zero linking with each of the $\gamma_i$ we have that $\widetilde{\eta}$ has zero linking with each of the $n$ lifts of each $\gamma_i$, hence $\widetilde{\eta}$ is null-homologous. 
\end{proof} 
Proposition~\ref{prop:generalmonotonicity} implies that the sequence $\{\di(\Sigma_2((\tau_{2k-1}J)(U)))\}$ is nondecreasing in $k$.  We also observe that Proposition~\ref{prop:generalmonotonicity} is not hard to extend to the case when $\Sigma_n(K)$ is not an integer homology sphere, but since it is not needed for our purposes we do not do so here.

\begin{example}\label{Example:-1twist}
Observe that when $k=0$, $\widetilde{\eta}$ has surgery coefficient $\infty$, and so we can delete it from the diagram of Figure~\ref{Fig:dbcsurgeryk} and with a bit of isotoping obtain the left side of Figure~\ref{Fig:-1twist}. Blowing down $\widetilde{\g_2}$ gives the surgery diagram on the right side of Figure~\ref{Fig:-1twist}. So $Y_0= S^3_1(5_2)$ and Corollary 1.5 of \cite{OS03alt} tells us that 
\begin{align*}\label{Equation:dinvariantof-1}
\di\left(\Sigma((\tau_{-1}P)(U))\right)= \di(S^3_1(5_2))= 2 \min \left\{0,  -\left\lceil\frac{-\sigma(5_2)}{4}\right\rceil\right\}
= 2 \min \left\{0,  -\left\lceil\frac{2}{4}\right\rceil\right\}=-2.
\end{align*}
\begin{figure}[h]
\includegraphics[height=4cm]{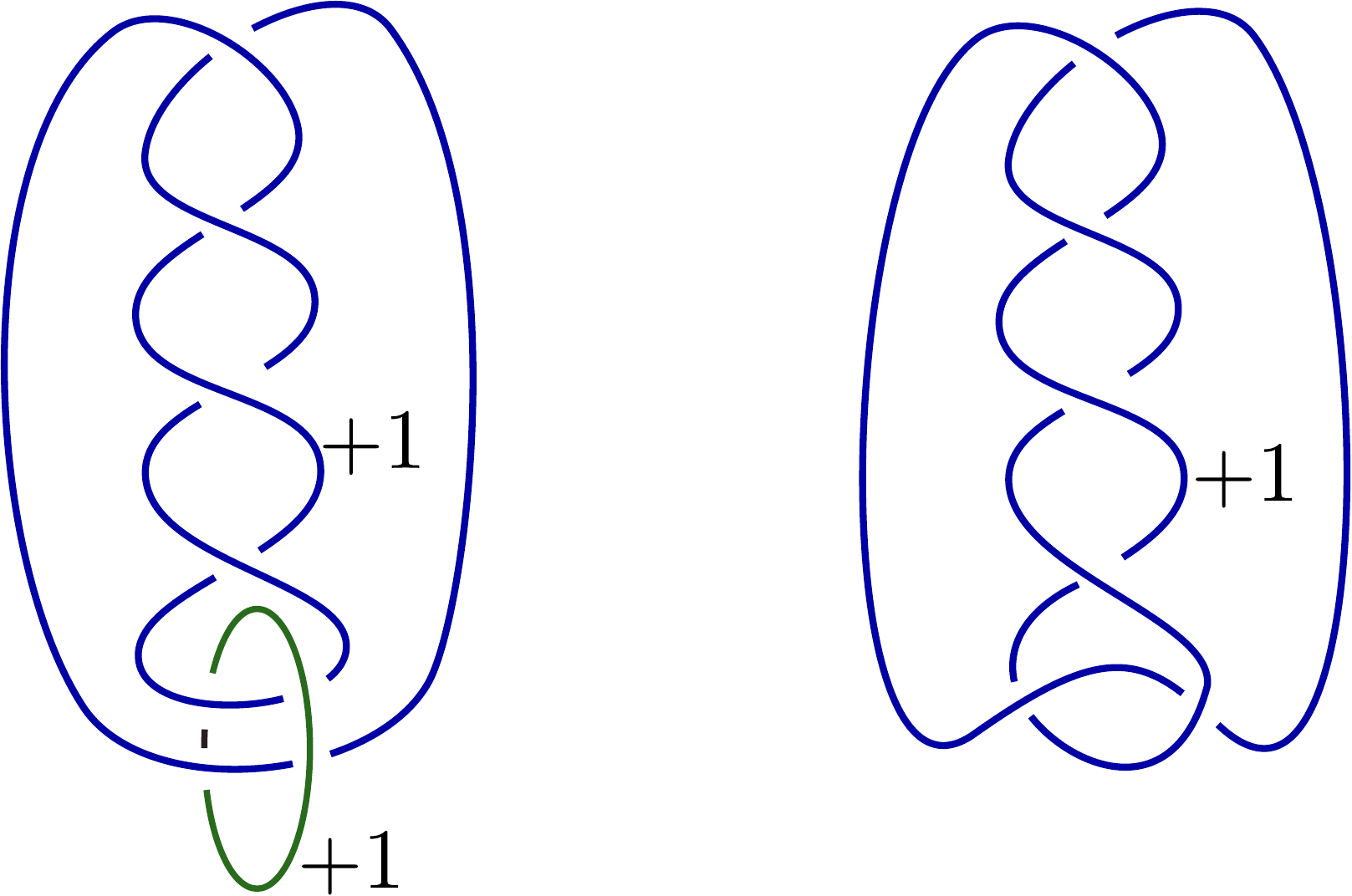}
\caption{Surgery diagrams for $Y_0= \Sigma_2((\tau_{-1}P)(U))$}
\label{Fig:-1twist}
\end{figure}
\end{example}

\begin{example}\label{Example:+1twist}
We use a new surgery diagram for $K_1$ to compute $\di(Y_1)$. 
The left side of Figure \ref{Fig:1twist} depicts $K_1$ in the standard $S^3$ and the middle gives a surgery diagram for $S^3$ in which $K_1$ appears unknotted. An isotopy takes $K$ to a curve with no self-crossings as on the right of Figure~\ref{Fig:1twist}. 
\begin{figure}[h]

\begin{tabular}{@{}c@{}}{\includegraphics[height=4.5cm]{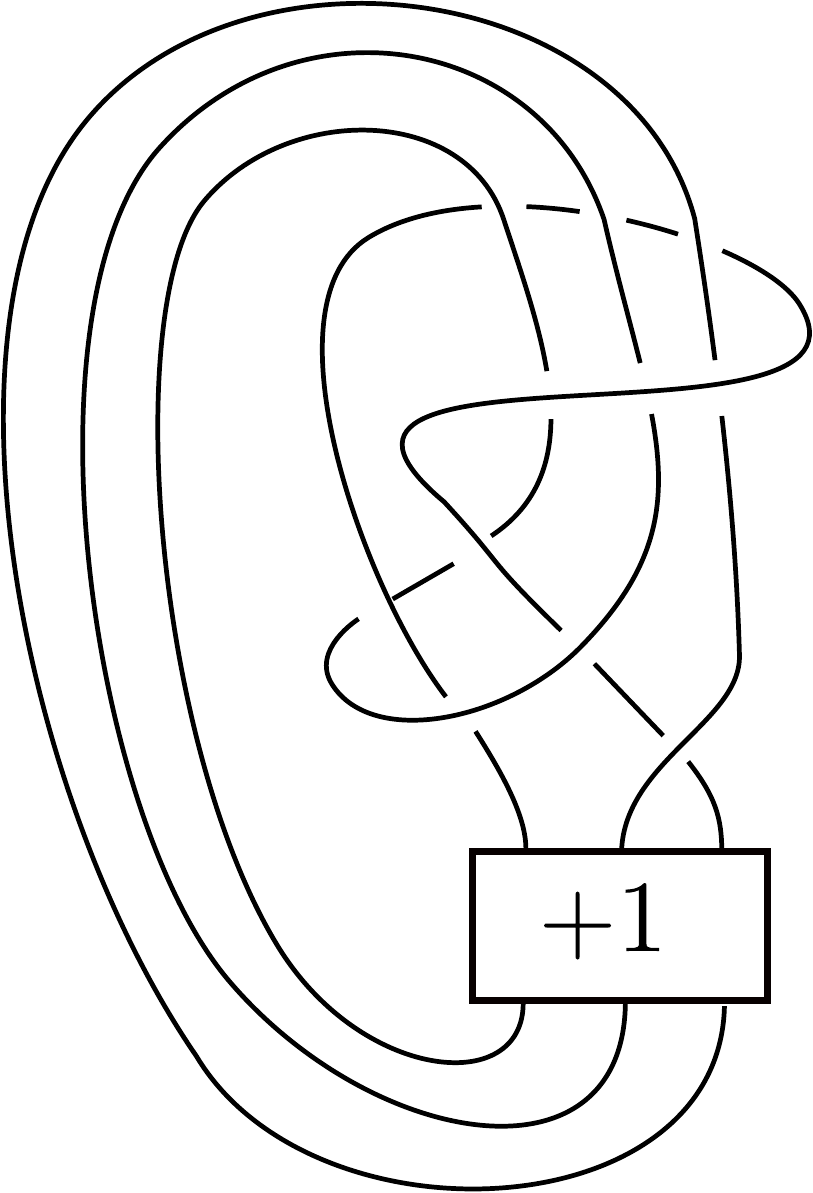}}\end{tabular}
\quad \quad
\begin{tabular}{@{}c@{}}{\includegraphics[height=4.5cm]{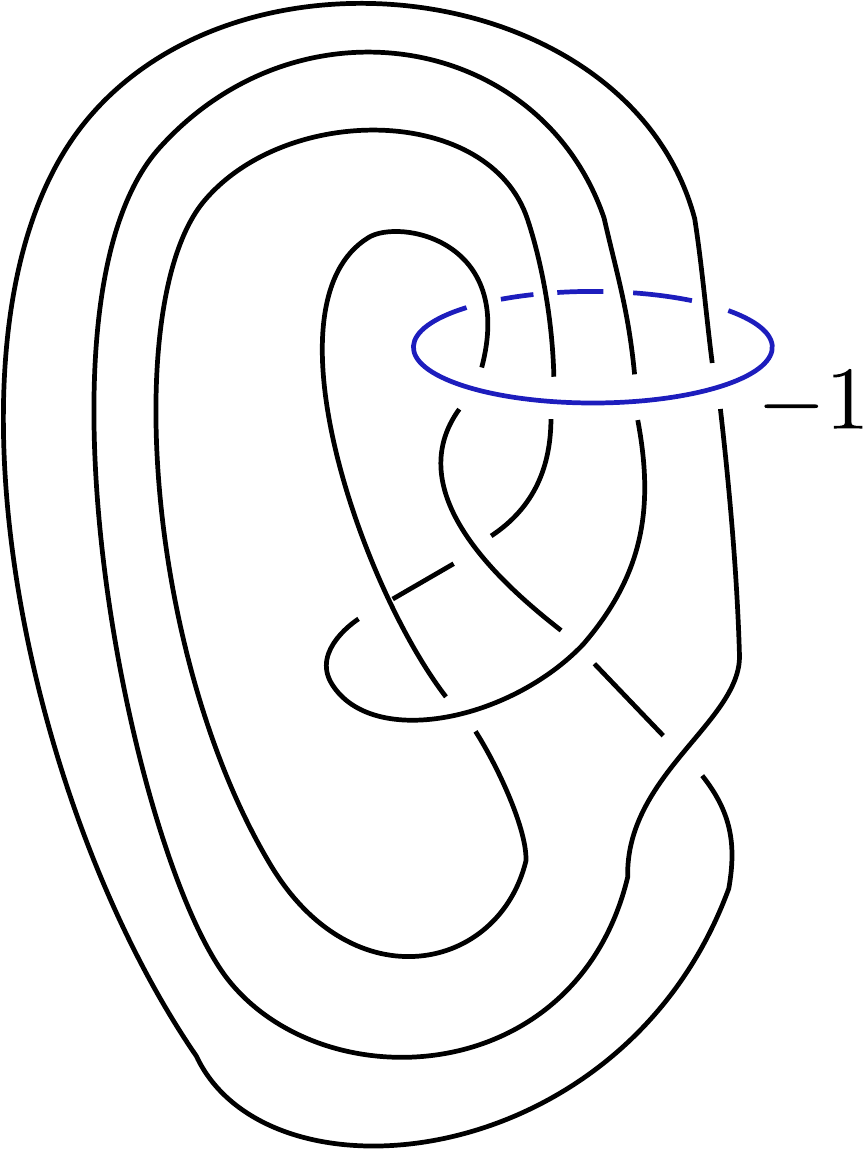}}\end{tabular}
\quad \quad
\begin{tabular}{@{}c@{}}{\includegraphics[height=3.75cm]{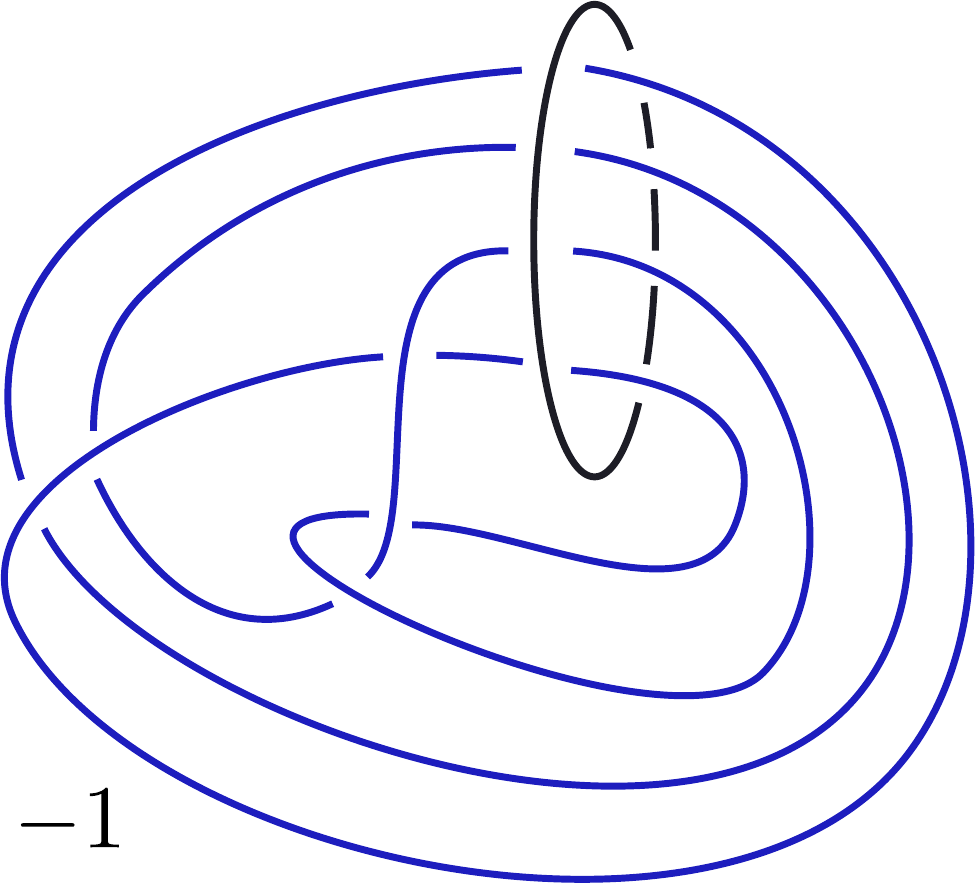}}\end{tabular}
\caption{The knot $(\tau_1 P)(U)$ (left) and two surgery diagrams for it (center, right).}
\label{Fig:1twist}
\end{figure}
We therefore have a surgery diagram for $Y_1$ as on the left of Figure~\ref{Fig:dbc+1}, which after some isotopy appears as in the center of Figure~\ref{Fig:dbc+1}.
Now observe that the 4-manifold obtained by adding one 0-framed 2-handle to $S^1\times B^3$ as on the right of Figure~\ref{Fig:dbc+1} has $\partial Z= \Sigma(K_1)$. 
Since $Z$ consists of a 0-handle and an algebraically canceling 1- and 2-handle pair, $Z$ is an integer homology ball. It follows from Theorem~1.2 of \cite{OS03abs} that $\di(\Sigma(K_1))=0$.
\begin{figure}[h]
\begin{tabular}{@{}c@{}}{\includegraphics[height=3cm]{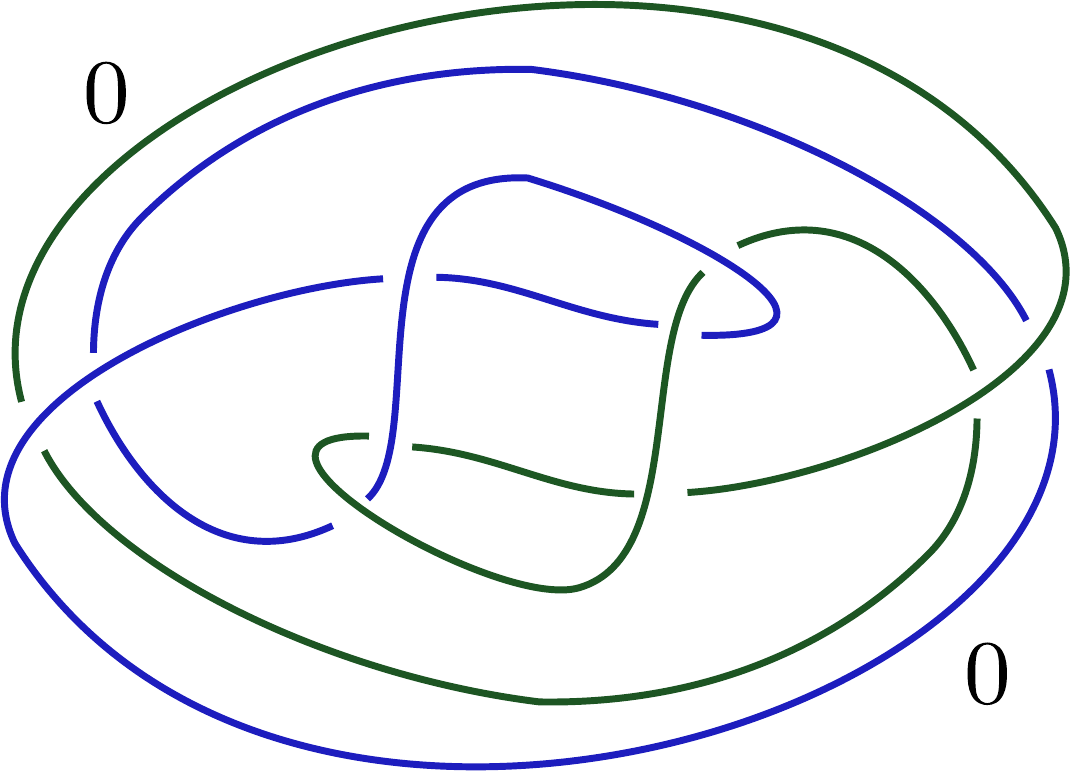}}\end{tabular}
\quad \quad 
\begin{tabular}{@{}c@{}}{\includegraphics[height=2.5cm]{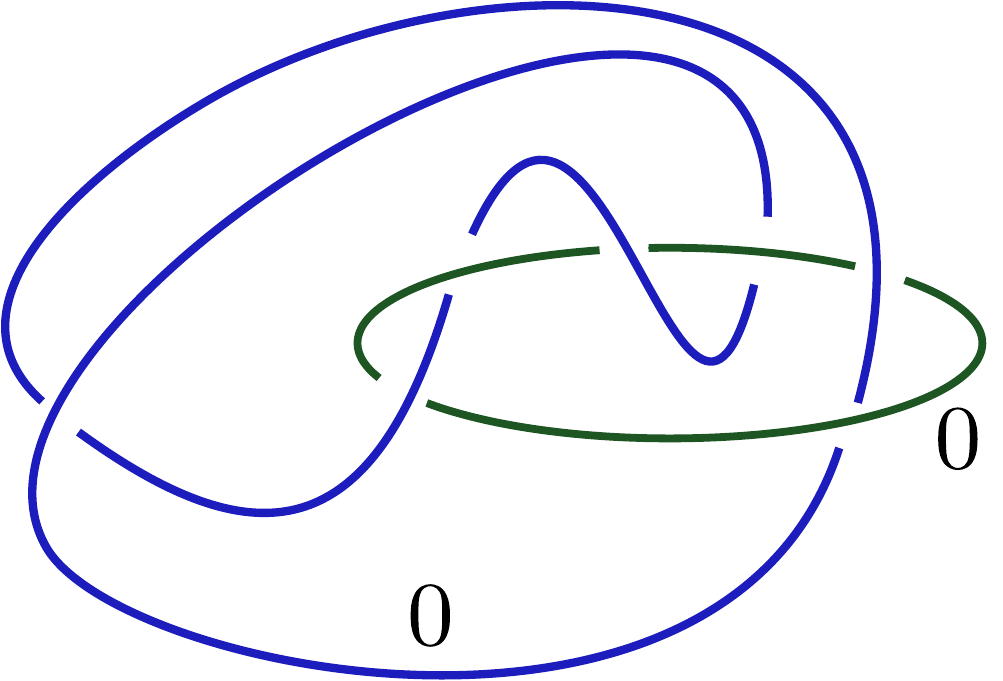}}\end{tabular}
\quad \quad 
\begin{tabular}{@{}c@{}}{\includegraphics[height=2.25cm]{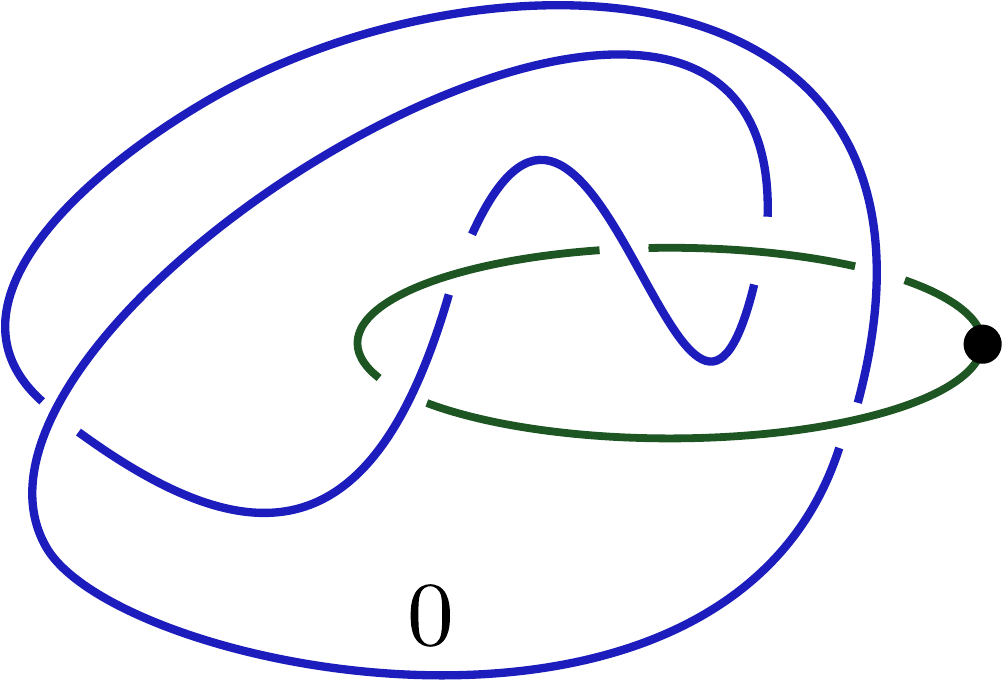}}\end{tabular}
\caption{Surgery diagrams for $\Sigma_2(K_1)$ (left, center) and a Kirby diagram for $Z^4$ with $\partial Z= \Sigma_2(K_1)$ (right).}
\label{Fig:dbc+1}
\end{figure}
\end{example}

\begin{proof}[(Theorem~\ref{thm:MainTheoremA})]
For any $k \in \N$, as before let $K_k= (\tau_{2k-1}J)(U)$. We now let $K_k'=(\tau_{-2k-3}J)(U)$.  
We will see in Section~\ref{section:interestingfamily} that when $k \neq k'$ an Alexander polynomial computation shows that  $K_k$ and $K_{k'}$ are distinct knots.

By Example~\ref{Ex:J}, Proposition~\ref{Prop:twist}, and Theorem~\ref{Thm:diffeotraces}, we have that $K_k$ and $K_k'$ have diffeomorphic 0-traces for any $k \in \N$. 
However, by the propositions and examples of this section, we have that 
\[ \di(\Sigma_2(K_k')) \leq \di(\Sigma_2(K_0))=-2 <0= \di(\Sigma_2(K_1)) \leq \di(\Sigma_2(K_k)), k \in \N.\] 
Therefore by Theorem~1.2 of \cite{OS03abs}, $\Sigma_2(K_k')$ and $\Sigma_2(K_k)$ are not smoothly rationally homology cobordant.  Thus $K_k'$ and $K_k$ are not smoothly concordant. 
\end{proof}

\begin{proof}[(Theorem \ref{thm:MainTheoremC})]
Observe that if a pattern $P$ acts by connected sum it must act by connected sum with $P(U)$ and so $P(U)$ and $P^*(U)$ must be concordant, since
\[U \sim P(P^{-1}(U)) \sim P(U) \# P^{-1}(U) \sim P(U) \#  \overline{P^*}(U).\]
But the examples in the proof of Theorem~\ref{thm:MainTheoremA} give us infinitely many dualizable, hence invertible, patterns for which this is not the case. 
\end{proof}

\begin{remark}\label{remark:puslice}
In fact, for any dualizable pattern $P$ with $P(U) \not \sim P^*(U)$, we can construct a dualizable pattern $Q$ such that $Q(U)$ is slice but the action of $Q$ on the smooth concordance group is nontrivial as follows. Let $-P_{\#}$ be the geometric winding number one pattern with $-P_{\#}(U)=-P(U)$ and let $Q= -P_{\#}\circ P$. Then $Q$ is a dualizable pattern with $Q(U) = P(U) \# - P(U) \sim U$. However, 
\[ Q(-P^*(U)) = P(-P^*(U)) \# -P(U) \sim - P(U) \not \sim -P^*(U).\] Note that this stands in stark contrast to the topological setting, in which it is still unknown whether there exist \emph{any} algebraic winding number one patterns $S$ with $S(U)$ topologically slice and yet with $S(K)$ not always topologically concordant to $K$. 

Applying Corollary \ref{Cor:sumtrace} to such a pattern $Q$ we get that for any non-trivial knot $K$ there exists $K'= K \# Q^*(U)$ with $K\sim K'$ and  a knot $K''= Q(K) \not \simeq K'$ such that $X_0(K')\cong X_0(K'')$.  Note that the fact that $K \# Q^*(U) \not \simeq Q(K)$ follows from the uniqueness of satellite descriptions of knots.  This leads us to the following question. 
\end{remark}

\begin{question}
For which concordance classes $[K]$ does there exist $[K']\neq [K]$ with representatives $K_0$ and $K_0'$ with diffeomorphic 0-traces?
\end{question}
Theorem \ref{thm:MainTheoremA} tells us that such concordance classes exist, and Corollary \ref{Cor:slicetrace} tells us that $[K]$ cannot be trivial. Further, as noted in the introduction, for a given $[K]$ there are many restrictions on the possible choices of $[K']$. One might also ask how many distinct $[K']$ are possible: can one find infinitely many concordance classes of knots all of which have representatives sharing a 0-trace?\\

\section{Relationship with annulus twisting.}\label{section:annulustwisting}

We begin by recalling a construction of certain knots in $S^3$. On the left side of Figure~\ref{Fig:simpleannulus}, ignoring for a moment the blue curve and the 0-surgery instruction, orient the black link in $S^3$ so that it represents the oriented boundary of the annulus it bounds in the figure. Then construct a knot $K$ by banding the link to itself in any way which preserves these orientations (the bold arcs represent attaching regions for the band). Knots arising from this construction are said to admit \emph{annulus presentations}. 
\begin{figure}[h]
\includegraphics[height=4.5cm]{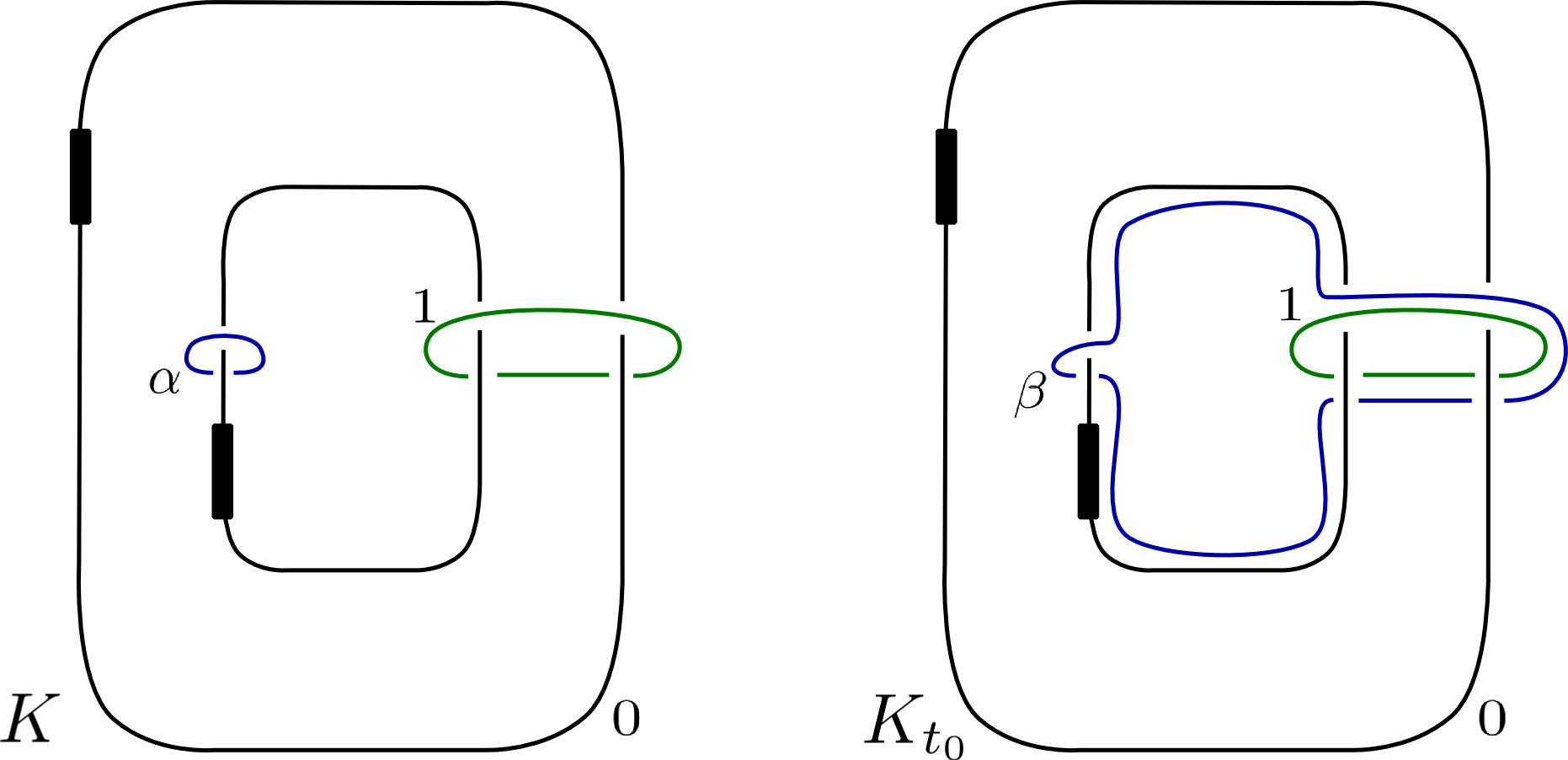}
\caption{Schematic diagrams of the manifolds obtained from 0-surgery on $K$ (left) and $K_{t_0}$ (right).}
\label{Fig:simpleannulus}
\end{figure}
We recall two relevant results. 
\begin{theorem}[(\cite{Oso06})]
If $K$ admits an annulus presentation then  for each $t \in \Z$ there is a knot $K_t$ also admitting an annulus presentation and such that $S^3_0(K)\cong S^3_0(K_t)$. 
\end{theorem}
In fact, Osoinach's result can be extended to give the following.
\begin{theorem}[(\cite{AJOT13})]
If $K$ admits an annulus presentation then Osoinach's homeomorphism $\phi:S^3_0(K)\to S^3_0(K_t)$ extends to a diffeomorphism $\Phi:X_0(K)\to X_0(K_t)$
\end{theorem}

Inspection of Osoinach's proof shows that there is some $t_0\in \{-1,1\}$ such that the homeomorphism $\phi:S^3_0(K)\to S^3_0(K_{t_0})$ induces a homeomorphism of pairs $\phi:(S^3_0(K),\alpha)\to (S^3_0(K_{t_0}),\beta)$ where $\alpha$ and $\beta$ are the knots in the surgery manifolds given in Figure \ref{Fig:simpleannulus} in blue. (cf. \cite{Oso06}, \cite{AJOT13}). 

The following exhibits a relationship between knots admitting annulus presentations and knots arising from dualizable patterns. 

\begin{proposition}
Let $K$ be a knot admitting an annulus presentation and $K_{t_0}$ be obtained from $K$ as above. Then there is a dualizable pattern $P$ such that $P(U) \simeq K_{t_0}$ and $P^*(U) \simeq K$. 
\label{Prop:annulusdual}
\end{proposition}

\begin{proof}
Define the link $(K_{t_0}, \beta ')$ in $S^3$ to be the link on the left hand side of Figure \ref{Fig:Pisdual}. 
\begin{figure}[h]
\includegraphics[height=4.5cm]{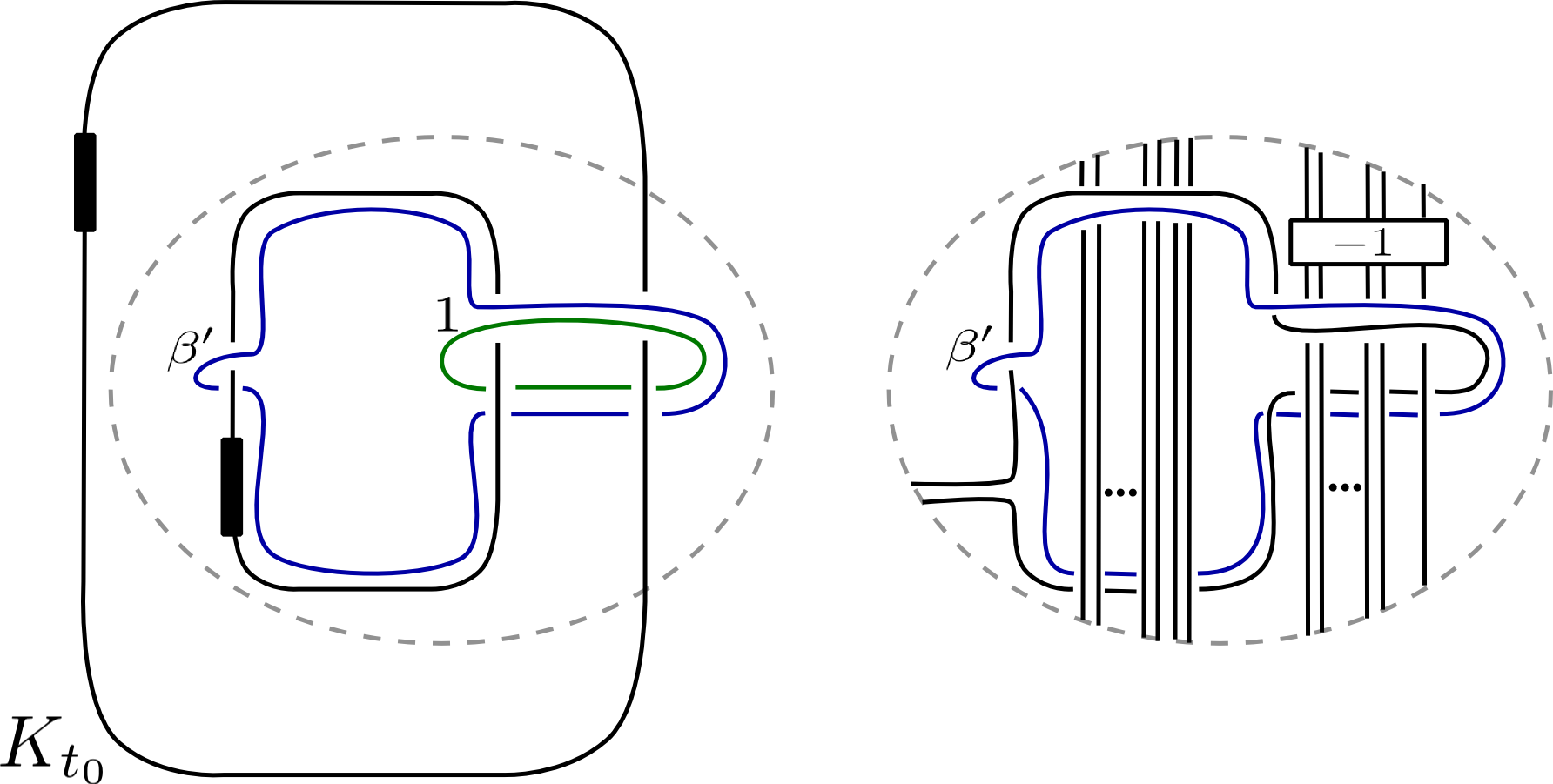}
\caption{$(K_{t_0}, \beta')$ is the result of an annulus twist (left), with schematic depiction of $K_{t_0}$ near $\beta'$ (right).}
\label{Fig:Pisdual}
\end{figure}
Consider the pattern $P$ obtained by restricting $K_{t_0}$ to the (standard) solid torus $V=S^3 \smallsetminus \nu(\beta ')$, which of course has the property that $P(U)\simeq K_{t_0}$. We claim that $P$ is dualizable.  By Proposition \ref{Prop:examples of dualizable patterns} we know that $P$ is dualizable if and only if the link $(K_{t_0}, \beta ')$ can be made isotopic to the Hopf link by sliding $K_{t_0}$ over a 0-framed $\beta '$ finitely many times. To see that $(K_{t_0}, \beta ')$ satisfies this, consider the right hand side of Figure \ref{Fig:Pisdual} in which we have drawn $(K_{t_0}, \beta')$ locally. We see that a single slide of $P$ over $\beta'$ takes $P$ to a core of the solid torus defined by $\beta'$. Therefore $P$ is dualizable and there exists some dual pattern $P^*$. 

Since we have a homeomorphism of pairs $\phi:(S^3_0(K),\alpha)\to(S^3_0(K_{t_0}),\beta)$ we have $S^3\smallsetminus\nu(K)\cong S^3_0(K) \smallsetminus \nu(\alpha)\cong S^3_0(K_{t_0}) \smallsetminus \nu(\beta)$. From Definition~\ref{Defn:dualpattern} we have that  $S^3_0(K_{t_0}) \smallsetminus \nu(\beta)$ is homeomorphic to $V\smallsetminus \nu(P)$ Dehn filled along $\lambda_P$, which is homeomorphic, via a map which preserves orientations, to  $S^3 - \nu(P^*(U))$. Therefore $S^3\smallsetminus\nu(K)\cong S^3 \smallsetminus \nu(P^*(U))$, and by the Knot Complement Theorem of \cite{GL89} we conclude that $K\simeq P^*(U)$. 
\end{proof}
It is not too hard to show that the above homeomorphism $f:S^3\smallsetminus\nu(K)\to S^3 \smallsetminus \nu(P^*(U))$ takes meridians to meridians if one prefers not to appeal to \cite{GL89}, but for the sake of brevity we do not pursue that here. 

Observe in Figure \ref{Fig:annulusexample} that each of the knots $(\tau_nJ)(U)$, where $J$ is the pattern from Example \ref{Exl:Gompf Miyazaki example}, admits a annulus presentation. 
\begin{figure}[h]
\includegraphics[height=4cm]{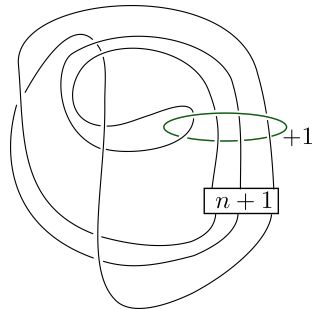}
\caption{The knot $(\tau_nJ)(U)$ admits an annulus presentation.}
\label{Fig:annulusexample}
\end{figure}
These examples, together with Theorem~\ref{thm:MainTheoremA}, give the following corollary to Proposition \ref{Prop:annulusdual}. 
\begin{corollary}\label{cor:annulustrace}
There are infinitely many knots $K$ and $K'$ such that $K$ admits an annulus presentation and $K'$ is obtained from $K$ by an annulus twist but such that $K$ and $K'$ are not smoothly concordant. 
\end{corollary}

\section{An interesting family of knots.}\label{section:interestingfamily}

We close by considering the family of knots $\{ (\tau_nJ)(U)\}_{n \in \Z}$, where 
$J$ is the pattern from Example \ref{Exl:Gompf Miyazaki example}. An explicit computation with Seifert matrices which we omit shows that, for $n\geq -2$,
\[ 
\Delta_{(\tau_nJ)(U)}(t)=
(t-1)^2(t^{2n+4}+1)+(2t^2-3t+2)t^{n+2}
\] Since the Alexander polynomial of $K$ is an invariant of the 0-surgery of $K$, Theorem \ref{Thm:diffeosurgeries} and Proposition \ref{Prop:twist} allow us to compute $\Delta_{(\tau_nJ)(U)}(t)$ for all $n\in\Z$. This shows that $(\tau_nJ)(U)$ is distinct from $(\tau_{n'}J)(U)$ for $n'\not\in \{n, -4-n\}$. It certainly follows from the proof of  Theorem \ref{thm:MainTheoremA} that this remains true when $n'=-4-n$ and $n$ is odd. We expect that when $n'=-4-n$ and $n$ is even we will still have $(\tau_nJ)(U)\not\simeq (\tau_{n'}J)(U)$ but do not pursue that here.

What makes this family remarkable is that for any pair  of integers $n$ and $n'$ Example \ref{Exl:Gompf Miyazaki example} and Theorem \ref{Thm:diffeontraces} show that there is an integer $r$  such that $X_r((\tau_nJ)(U))\cong X_r((\tau_{n'}J)(U))$. In fact, for any $n$ there is at most a single integer $s$ such that $X_s((\tau_nJ)(U))$ is not diffeomorphic to $X_s((\tau_mJ)(U))$ for some $m\in \Z$. 

We also point out that all the knots in this family have four genus equal to one. To show the four genera are all one or less we point out that for any $n$ there is a band move taking $(\tau_nJ)(U)$ to a Hopf link. For $n$ even, we can compute that the determinant of $(\tau_nJ)(U)$ equals 15 and hence $(\tau_nJ)(U)$ is not slice. For $n$ odd, we note that by Theorem~\ref{Thm:diffeotraces} and Corollary~\ref{Cor:slicetrace} it suffices to show that $(\tau_nJ)(U)$ is not slice for $n\leq -1$. However, we will see in Proposition~\ref{prop:monotonicity} and Example~\ref{Example:-1twist} that when $n \leq -1$ is odd we have
\[ d(\Sigma_2((\tau_nJ)(U)) \leq d(\Sigma_2((\tau_{-1}J)(U))) =-2<0,\]
and it follows that $(\tau_nJ)(U)$ is not slice. 

In fact, we have $g_4((\tau_{n}J)(U)\# - (\tau_{n'}J)(U)) \leq 1$ for any pair of integers $n$ and $n'$, as shown by the four band moves shown in Figure~\ref{Fig:fourgenusbound}, which take $(\tau_{n}J)(U)\# - (\tau_{n'}J)(U)$ to a three component unlink. 
\begin{figure}[h]
\includegraphics[height=4cm]{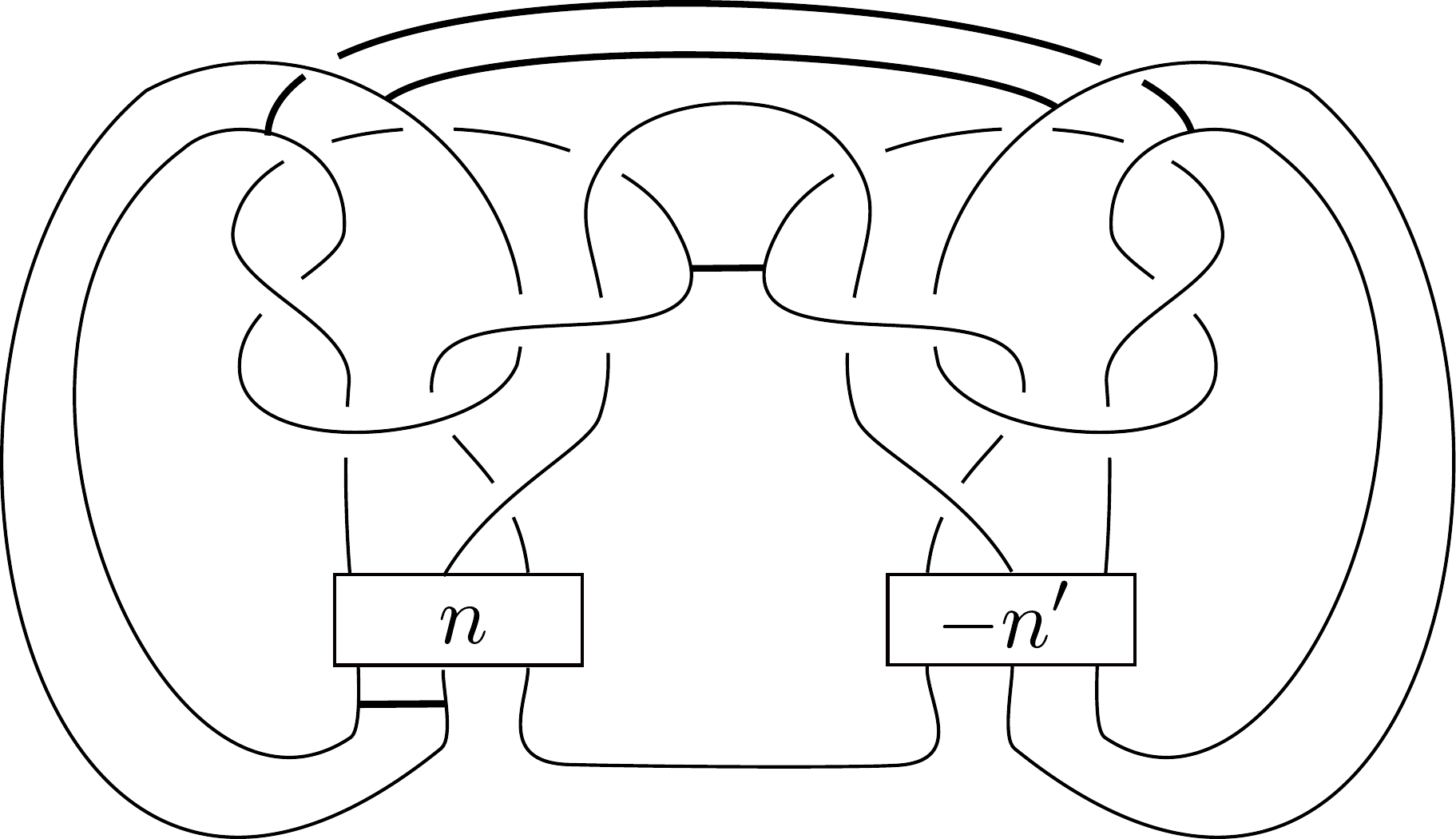}
\caption{The four genus of $(\tau_{n}J)(U)\# - (\tau_{n'}J)(U)$ is no more than 1.}
\label{Fig:fourgenusbound}
\end{figure}
(In fact, one can show that $g_4((\tau_{n}J)(K)\# - (\tau_{n'}J)(U)\#-K) \leq 1$ by a virtually identical argument.) Since for odd $n \neq n'$ we have that $(\tau_{n}J)(U)$ and $(\tau_{n'}J)(U)$ are not concordant, this gives an infinite family of knots any two of which are distance exactly 1 from each other under the metric $d(K,J):= g_4(K \#-J)$. 

\bibliography{satellite}

\affiliationone{
   A. Miller and L. Piccirillo\\
   University of Texas at Austin\\
      Department of Mathematics\\
      2515 Speedway. Austin, TX 78712
   USA
   \email{amiller@math.utexas.edu\\
   lpiccirillo@math.utexas.edu}}
\end{document}